\documentclass{amsart}

\usepackage[normalem]{ulem}
\usepackage{mathtools}
\usepackage{mathrsfs}  
\usepackage[arrow,curve,matrix]{xy}

\newif\ifdraft
\draftfalse

\DeclareFontFamily{OMS}{rsfs}{\skewchar\font'60}
\DeclareFontShape{OMS}{rsfs}{m}{n}{<-5>rsfs5 <5-7>rsfs7 <7->rsfs10 }{}
\DeclareSymbolFont{rsfs}{OMS}{rsfs}{m}{n}
\DeclareSymbolFontAlphabet{\scr}{rsfs}

\numberwithin{figure}{section}

\newcommand{\bbC}{\mathbb{C}}

\newcommand{\N}{\mathbb{N}}

\newcommand{\Q}{\mathbb{Q}}

\newcommand{\V}{\mathbb{V}}

\newcommand{\Z}{\mathbb{Z}}

\newcommand{\cB}{\mathcal{B}}

\newcommand{\cH}{\mathcal{H}}

\newcommand{\cM}{\mathcal{M}}

\newcommand{\bQ}{\mathbf{Q}}
\newcommand{\bR}{\mathbf{R}}

\newcommand{\fddot}{F_\bullet}
\newcommand{\rarr}{\longrightarrow}

\newcommand{\Exc}{\mathrm{Exc}}
\newcommand{\grf}{\gr^F_\bullet}

\newcommand{\GF}{\shG_\bullet}
\newcommand{\FF}{\shF_\bullet}

\newcommand{\sP}{\mathscr{P}}
\newcommand{\sM}{\mathscr{M}}
\newcommand{\sB}{\mathscr{B}}
\newcommand{\shA}{\mathscr{A}}
\newcommand{\shTA}{\mathscr{T}}
\newcommand{\shD}{\mathscr{D}}

\theoremstyle{plain}
\newtheorem{theorem}{Theorem}[section]

\newtheorem{proposition}[theorem]{Proposition}
\newtheorem{corollary}[theorem]{Corollary}
\newtheorem{lemma}[theorem]{Lemma}

\theoremstyle{definition}
\newtheorem{definition}[theorem]{Definition}

\theoremstyle{remark}
\newtheorem{remark}[theorem]{Remark}

\newtheorem{set-up}[theorem]{Set-up}
\newtheorem{claim}[theorem]{Claim}

\usepackage{amsmath,amsfonts,amsthm,mathrsfs}
\usepackage{amssymb}
\usepackage{mathrsfs}

\usepackage[bookmarks]{hyperref}
\usepackage[usenames,dvipsnames]{xcolor}
\hypersetup{colorlinks=true,citecolor=NavyBlue,linkcolor=BrickRed,urlcolor=Orange}

\usepackage[alphabetic,initials]{amsrefs}

\usepackage{enumitem}

\usepackage{chngcntr}

\ifdraft
\usepackage[notcite,notref,color]{showkeys}

\definecolor{labelkey}{gray}{0.5}
\fi

\usepackage{tikz}

\usepackage{tikz-cd}
\tikzset{commutative diagrams/arrow style=math font}

\usetikzlibrary{matrix,arrows}
\newlength{\myarrowsize} 

\pgfarrowsdeclare{cmto}{cmto}{
	\pgfsetdash{}{0pt} 
	\pgfsetbeveljoin 
	\pgfsetroundcap 
	\setlength{\myarrowsize}{0.6pt}
	\addtolength{\myarrowsize}{.5\pgflinewidth}
	\pgfarrowsleftextend{-4\myarrowsize-.5\pgflinewidth} 
	\pgfarrowsrightextend{.8\pgflinewidth}
}{
	\setlength{\myarrowsize}{0.6pt} 
  	\addtolength{\myarrowsize}{.5\pgflinewidth}  
	\pgfsetlinewidth{0.5\pgflinewidth}
	\pgfsetroundjoin
	\pgfpathmoveto{\pgfpoint{1.5\pgflinewidth}{0}}
	\pgfpatharc{-109}{-170}{4\myarrowsize}
	\pgfpatharc{10}{189}{0.58\pgflinewidth and 0.2\pgflinewidth}
	\pgfpatharc{-170}{-115}{4\myarrowsize+\pgflinewidth}
	\pgfpathclose
	\pgfusepathqfillstroke
	\pgfpathmoveto{\pgfpoint{1.5\pgflinewidth}{0}}
	\pgfpatharc{109}{170}{4\myarrowsize}
	\pgfpatharc{-10}{-189}{0.58\pgflinewidth and 0.2\pgflinewidth}
	\pgfpatharc{170}{115}{4\myarrowsize+\pgflinewidth}
	\pgfpathclose
	\pgfusepathqfillstroke
	\pgfsetlinewidth{2\pgflinewidth}
}

\pgfarrowsdeclare{cmonto}{cmonto}{
	\pgfsetdash{}{0pt} 
	\pgfsetbeveljoin 
	\pgfsetroundcap 
	\setlength{\myarrowsize}{0.6pt}
	\addtolength{\myarrowsize}{.5\pgflinewidth}
	\pgfarrowsleftextend{-4\myarrowsize-.5\pgflinewidth} 
	\pgfarrowsrightextend{.8\pgflinewidth}
}{
	\setlength{\myarrowsize}{0.6pt} 
  	\addtolength{\myarrowsize}{.5\pgflinewidth}  
	\pgfsetlinewidth{0.5\pgflinewidth}
	\pgfsetroundjoin
	\pgfpathmoveto{\pgfpoint{1.5\pgflinewidth}{0}}
	\pgfpatharc{-109}{-170}{4\myarrowsize}
	\pgfpatharc{10}{189}{0.58\pgflinewidth and 0.2\pgflinewidth}
	\pgfpatharc{-170}{-115}{4\myarrowsize+\pgflinewidth}
	\pgfpathclose
	\pgfusepathqfillstroke
	\pgfpathmoveto{\pgfpoint{1.5\pgflinewidth}{0}}
	\pgfpatharc{109}{170}{4\myarrowsize}
	\pgfpatharc{-10}{-189}{0.58\pgflinewidth and 0.2\pgflinewidth}
	\pgfpatharc{170}{115}{4\myarrowsize+\pgflinewidth}
	\pgfpathclose
	\pgfusepathqfillstroke
	\pgfpathmoveto{\pgfpoint{1.5\pgflinewidth-0.3em}{0}}
	\pgfpatharc{-109}{-170}{4\myarrowsize}
	\pgfpatharc{10}{189}{0.58\pgflinewidth and 0.2\pgflinewidth}
	\pgfpatharc{-170}{-115}{4\myarrowsize+\pgflinewidth}
	\pgfpathclose
	\pgfusepathqfillstroke
	\pgfpathmoveto{\pgfpoint{1.5\pgflinewidth-0.3em}{0}}
	\pgfpatharc{109}{170}{4\myarrowsize}
	\pgfpatharc{-10}{-189}{0.58\pgflinewidth and 0.2\pgflinewidth}
	\pgfpatharc{170}{115}{4\myarrowsize+\pgflinewidth}
	\pgfpathclose
	\pgfusepathqfillstroke
	\pgfsetlinewidth{2\pgflinewidth}
}

\pgfarrowsdeclare{cmhook}{cmhook}{
	\pgfsetdash{}{0pt} 
	\pgfsetbeveljoin 
	\pgfsetroundcap 
	\setlength{\myarrowsize}{0.6pt}
	\addtolength{\myarrowsize}{.5\pgflinewidth}
	\pgfarrowsleftextend{-4\myarrowsize-.5\pgflinewidth} 
	\pgfarrowsrightextend{.8\pgflinewidth}
}{
	\setlength{\myarrowsize}{0.6pt} 
  	\addtolength{\myarrowsize}{.5\pgflinewidth}  
 	\pgfsetdash{}{0pt}
	\pgfsetroundcap
	\pgfpathmoveto{\pgfqpoint{0pt}{-4.667\pgflinewidth}}
	\pgfpathcurveto
    {\pgfqpoint{4\pgflinewidth}{-4.667\pgflinewidth}}
    {\pgfqpoint{4\pgflinewidth}{0pt}}
    {\pgfpointorigin}
	\pgfusepathqstroke
}

\newenvironment{diagram*}[2]{%
\[%
\begin{tikzpicture}[>=cmto,baseline=(current bounding box.center),%
	to/.style={->,font=\scriptsize,cap=round},%
	into/.style={cmhook->,font=\scriptsize,cap=round},%
	onto/.style={-cmonto,font=\scriptsize,cap=round},%
	math/.style={matrix of math nodes, row sep=#2, column sep=#1,%
		text height=1.5ex, text depth=0.25ex}]%
}{%
\end{tikzpicture}%
\]%
\ignorespacesafterend%
}

\newcommand{\Dmod}{\mathscr{D}}

\newcommand{\ZZ}{\mathbb{Z}}
\newcommand{\QQ}{\mathbb{Q}}

\newcommand{\CC}{\mathbb{C}}

\newcommand{\PP}{\mathbb{P}}

\renewcommand{\Im}{\operatorname{Im}}

\DeclareMathOperator{\Sym}{Sym}
\DeclareMathOperator{\gr}{gr}

\newcommand{\shf}[1]{\mathscr{#1}}

\newcommand{\sA}{\scr{A}}
\newcommand{\sC}{\scr{C}}
\newcommand{\sD}{\scr{D}}
\newcommand{\sE}{\scr{E}}
\newcommand{\sL}{\scr{L}}
\newcommand{\sF}{\scr{F}}
\newcommand{\sG}{\scr{G}}
\newcommand{\sH}{\scr{H}}
\newcommand{\sK}{\scr{K}}
\newcommand{\sN}{\scr{N}}
\newcommand{\sO}{\scr{O}}
\newcommand{\sT}{\scr{T}}

\def\overbar#1#2#3{{%
	\setbox0=\hbox{\displaystyle{#1}}%
	\dimen0=\wd0
	\advance\dimen0 by -#2 
	\vbox {\nointerlineskip \moveright #3 \vbox{\hrule height 0.3pt width \dimen0}%
		\nointerlineskip \vskip 1.5pt \box0}%
}}

\newcommand{\shF}{\shf{F}}
\newcommand{\shG}{\shf{G}}

\newcommand{\shO}{\shf{O}}

\newcommand{\wtilde}{\widetilde}

\makeatletter
\let\@@seccntformat\@seccntformat
\renewcommand*{\@seccntformat}[1]{%
  \expandafter\ifx\csname @seccntformat@#1\endcsname\relax
    \expandafter\@@seccntformat
  \else
    \expandafter
      \csname @seccntformat@#1\expandafter\endcsname
  \fi
    {#1}%
}
\newcommand*{\@seccntformat@subsection}[1]{%
  \textbf{\csname the#1\endcsname.}
}
\makeatother

\makeatletter
\let\@paragraph\paragraph
\renewcommand*{\paragraph}[1]{%
	\vspace{0.3\baselineskip}%
	\@paragraph{\textit{#1}}%
}
\makeatother

\newcommand{\theoremref}[1]{\hyperref[#1]{Theorem~\ref*{#1}}}
\newcommand{\lemmaref}[1]{\hyperref[#1]{Lemma~\ref*{#1}}}
\newcommand{\definitionref}[1]{\hyperref[#1]{Definition~\ref*{#1}}}
\newcommand{\propositionref}[1]{\hyperref[#1]{Proposition~\ref*{#1}}}
\newcommand{\conjectureref}[1]{\hyperref[#1]{Conjecture~\ref*{#1}}}
\newcommand{\corollaryref}[1]{\hyperref[#1]{Corollary~\ref*{#1}}}
\newcommand{\exampleref}[1]{\hyperref[#1]{Example~\ref*{#1}}}

\makeatletter
\let\old@caption\caption
\renewcommand*{\caption}[1]{%
	\setcounter{figure}{\value{equation}}%
	\stepcounter{equation}%
	\old@caption{#1}\relax%
}
\makeatother

\newcounter{intro}

\newtheorem{intro-conjecture}[intro]{Conjecture}
\newtheorem{intro-corollary}[intro]{Corollary}
\newtheorem{intro-theorem}[intro]{Theorem}

\newcommand{\parref}[1]{\hyperref[#1]{\S\ref*{#1}}}

\makeatletter
\newcommand*\if@single[3]{%
  \setbox0\hbox{${\mathaccent"0362{#1}}^H$}%
  \setbox2\hbox{${\mathaccent"0362{\kern0pt#1}}^H$}%
  \ifdim\ht0=\ht2 #3\else #2\fi
  }
\newcommand*\rel@kern[1]{\kern#1\dimexpr\macc@kerna}
\newcommand*\widebar[1]{\@ifnextchar^{{\wide@bar{#1}{0}}}{\wide@bar{#1}{1}}}
\newcommand*\wide@bar[2]{\if@single{#1}{\wide@bar@{#1}{#2}{1}}{\wide@bar@{#1}{#2}{2}}}
\newcommand*\wide@bar@[3]{%
  \begingroup
  \def\mathaccent##1##2{%
    \if#32 \let\macc@nucleus\first@char \fi
    \setbox\z@\hbox{$\macc@style{\macc@nucleus}_{}$}%
    \setbox\tw@\hbox{$\macc@style{\macc@nucleus}{}_{}$}%
    \dimen@\wd\tw@
    \advance\dimen@-\wd\z@
    \divide\dimen@ 3
    \@tempdima\wd\tw@
    \advance\@tempdima-\scriptspace
    \divide\@tempdima 10
    \advance\dimen@-\@tempdima
    \ifdim\dimen@>\z@ \dimen@0pt\fi
    \rel@kern{0.6}\kern-\dimen@
    \if#31
      \overline{\rel@kern{-0.6}\kern\dimen@\macc@nucleus\rel@kern{0.4}\kern\dimen@}%
      \advance\dimen@0.4\dimexpr\macc@kerna
      \let\final@kern#2%
      \ifdim\dimen@<\z@ \let\final@kern1\fi
      \if\final@kern1 \kern-\dimen@\fi
    \else
      \overline{\rel@kern{-0.6}\kern\dimen@#1}%
    \fi
  }%
  \macc@depth\@ne
  \let\math@bgroup\@empty \let\math@egroup\macc@set@skewchar
  \mathsurround\z@ \frozen@everymath{\mathgroup\macc@group\relax}%
  \macc@set@skewchar\relax
  \let\mathaccentV\macc@nested@a
  \if#31
    \macc@nested@a\relax111{#1}%
  \else
    \def\gobble@till@marker##1\endmarker{}%
    \futurelet\first@char\gobble@till@marker#1\endmarker
    \ifcat\noexpand\first@char A\else
      \def\first@char{}%
    \fi
    \macc@nested@a\relax111{\first@char}%
  \fi
  \endgroup
}
\makeatother

\setlist[enumerate]{label=(\thetheorem.\arabic*), before={\setcounter{enumi}{\value{equation}}}, after={\setcounter{equation}{\value{enumi}}}}

\numberwithin{equation}{theorem}

\begin{document}

\title[Brody hyperbolicity  of base spaces]{Brody hyperbolicity  of base spaces of certain families of varieties}

\author{Mihnea Popa}
\address{Mihnea Popa, Department of Mathematics, Northwestern University,
2033 Sheridan Road, Evanston, IL 60208, USA} 
\email{\href{mailto:mpopa@math.northwestern.edu}{mpopa@math.northwestern.edu}}
\urladdr{\href{http://www.math.northwestern.edu/~mpopa/}{http://www.math.northwestern.edu/~mpopa/}}

\author{Behrouz Taji}
\address{Behrouz Taji, University of Notre Dame, Department of Mathematics, 278 Hurley, Notre Dame, IN
46556 USA.}
\email{\href{mailto:btaji@nd.edu}{btaji@nd.edu}}
\urladdr{\href{http://sites.nd.edu/b-taji}{http://sites.nd.edu/b-taji}}

\author{Lei Wu}
\address{Lei Wu, Department of Mathematics, University of Utah,
155 S 1400 E, Salt Lake City, UT 84112, USA}
\email{\href{mailto:lwu@math.utah.edu}{lwu@math.utah.edu}}
\urladdr{\href{https://www.math.utah.edu/~lwu/}{https://www.math.utah.edu/~lwu/}}



\thanks{MP was partially supported by the NSF grant DMS-1700819}


\setlength{\parskip}{0.19\baselineskip}


\begin{abstract}
We prove that quasi-projective base spaces of smooth families of minimal varieties of general type 
with maximal variation do not admit Zariski dense entire curves. We deduce the fact that 
moduli stacks of polarized varieties of this sort are Brody hyperbolic, answering a special case of a question of
Viehweg and Zuo. For two-dimensional bases, we show analogous results in the more general case of families 
of varieties admitting a good minimal model.

\end{abstract}

\maketitle


\section{Introduction}
\label{sect:intro}

The purpose of this paper is to establish a few results related to the hyperbolicity of base spaces of families of 
smooth complex varieties having maximal variation. Our study is motivated by the conjecturally degenerate behavior of 
entire curves inside the moduli $P_h$ of polarized manifolds, corresponding to the moduli functor $\sP_h$ which associates to a variety
$V$ the set $\sP_h (V)$ of pairs $(f\colon U \to V, \sH)$, where $f$ is a smooth projective morphism whose fibers have semiample 
canonical bundle and $\sH$ is an $f$-ample line bundle with Hilbert polynomial $h$, up to isomorphisms and fiberwise numerical equivalence.
The coarse moduli spaces $P_h$ were shown to be quasi-projective schemes by Viehweg, cf.~\cite{Viehweg95}.

\subsection{Families of minimal varieties of general type}
The first result partially answers a question of Viehweg and Zuo, cf. \cite[Quest.~0.2]{Vie-Zuo03a}, 
who established in  their fundamental paper the analogous result in the case of moduli of canonically polarized manifolds 
(i.e. those whose canonical bundle is ample), cf.~\cite[Thm.~0.1]{Vie-Zuo03a}.

\begin{theorem}
\label{thm:BHMain}
Let $f_U \colon U \to V$ be a smooth family of polarized manifolds of \emph{general type} in $\sP_h (V)$, with $V$  
quasi-projective, such that the induced morphism $\sigma\colon V \to P_h$ is quasi-finite onto its image.
Then $V$ is Brody hyperbolic, that is any holomorphic map $\gamma \colon \mathbb C \to V$ is constant. 
\end{theorem}

The question in \cite{Vie-Zuo03a} asks whether the same holds for moduli of arbitrary polarized varieties, i.e. not 
necessarily of general type. While this was our original goal, in the general case we have not been able to overcome difficulties related to vanishing theorems.
We do however give a positive answer to an even more general version of this question when $V$ is a surface; see Corollary \ref{cor:MainBH}. Note that the more restrictive property of algebraic hyperbolicity, involving algebraic maps from curves and abelian varieties, has been known in great generality. It  was established by Kov\'acs \cite{Kovacs00a} for moduli of canonically polarized manifolds, and then by a combination of Viehweg-Zuo \cite{Vie-Zuo01} and Popa-Schnell \cite{PS15} for families admitting good minimal models. See also Migliorini \cite{Migliorini95} for families of surfaces.

Theorem~\ref{thm:BHMain} is a direct consequence of the following result regarding the 
base spaces of smooth families of minimal manifolds of general type that have 
maximal variation. 
Recall first that the exceptional locus of $V$ is defined as 
$$
\Exc(V)  : = \overline{\Big( \bigcup_{\gamma} \gamma(\mathbb C)\Big)},
$$
where the union is taken over all non-constant holomorphic maps $\gamma\colon \mathbb C \to V$, and the closure is in the Zariski topology.

\begin{theorem}\label{thm:MaxVar}
Let $f_U\colon U\to V$ be a smooth projective morphism of smooth, quasi-projective 
varieties. Assume that $f_U$ has maximal variation, and that its fibers are minimal 
manifolds of general type. Then the exceptional locus $\Exc(V)$ is a proper subset of $V$. 
In particular, every holomorphic map $\gamma: \mathbb C \to V$
is algebraically degenerate, that is the image of $\gamma$ is not Zariski dense. 
\end{theorem}

For the general definition of the variation ${\rm Var} (f)$ of a family, we refer to~\cite{Viehweg83}. We are only concerned 
with maximal variation, ${\rm Var} (f) = \dim V$, which means that the very general fiber can only be birational to countably many  other fibers; cf. also Lemma \ref{countable}.
For families coming from maps to moduli schemes,  maximal variation simply means that the moduli map $V\to M$  
is generically finite. 
 
 The theorem above is of course especially relevant for families of surfaces, where the minimality assumption becomes 
 unnecessary, as one can pass to smooth minimal models in families. Recall that Giesker \cite{Gieseker} has constructed a coarse moduli space $M$ parametrizing birational isomorphism classes of surfaces of general type.

\begin{corollary}
Let $f_U\colon U\to V$ be a smooth projective family of surfaces of general type with maximal variation. Then 
$\Exc(V)$ is a proper subset of $V$.  If moreover the family comes from a quasi-finite map $V \to M$ to the moduli space of surfaces of general type, then $V$ is Brody hyperbolic.
\end{corollary}

Statements as in Theorem~\ref{thm:BHMain} and~\ref{thm:MaxVar} are conjecturally expected to be consequences 
of a different property of a more algebraic flavor, which is the subject of Viehweg's hyperbolicity conjecture;
itself a generalization of a conjecture of Shafaravich. 
Roughly speaking, Viehweg predicted that for families with maximal variation, a log smooth compactification $(Y, D)$ of $V$ 
is of log general type. The proof of the original statement of the conjecture, in the canonically polarized case, was 
established in important special cases in \cite{VZ02}, \cite{KK08}, \cite{KK08b}, \cite{KK10}, \cite{MR2871152}, 
and was recently completed by Campana and P\u{a}un \cite[Thm.~8.1]{CP16}; for a more detailed overview of this body of 
work and for further references, please see \cite[\S1.2]{PS15}.
The statement was subsequently extended to families whose geometric generic fiber admits a good minimal model, so in particular to families of 
varieties of general type, by the first author and Schnell \cite[Thm.~A]{PS15}. 
On the other hand, the conjecture of Green-Griffiths-Lang, \cite{GG80}
and~\cite{Lan86},  predicts that for a pair $(Y, D)$ of log general type, the image of any entire 
curve $\gamma\colon \mathbb C \to V$ is algebraically degenerate, where $V = Y \smallsetminus D$.

In the canonically polarized case, the problem of hyperbolicity of moduli stacks has 
a rich history from the purely analytic point of view. 
For the moduli stack $\mathcal M_g$ of compact Riemann surfaces of genus $g$, results of Ahlfors
\cite{Ahl61}, Royden \cite{Roy74}
and Wolpert \cite{Wol86} show that the holomorphic sectional curvature of 
the Weil-Petersson metric on the base of a family admitting a quasi-finite map
to $\mathcal M_g$, with $g\geq 2$, is negative 
and bounded away from zero. In particular, such base spaces are Brody hyperbolic.
In higher dimensions, thanks to Aubin-Yau's 
solution to Calabi's conjecture, one studies equivalently 
families of compact complex manifolds admitting a smooth K\"ahler-Einstein metric with negative 
Ricci curvature. The first breakthrough in this direction was achieved by Siu's computation \cite{Siu86}
of the curvature of the Weil-Petersson metric on the moduli via the K\"ahler-Einstein metric of the fibers
of the family (see also~\cite{Sch12}).  To and Yeung \cite{TY15} built upon Siu's work to prove the Kobayashi hyperbolicity 
of moduli stacks of canonically polarized manifolds and thus
gave a new proof of the Brody hyperbolicity of such moduli stacks (see also~\cite{TY16} for the Ricci-flat case). 
We also refer the reader to~\cite[Thm.~9]{Sch17}.  A different proof of this result has been established by 
 Berndtsson, P\u{a}un and Wang \cite{BPW}. Recently, based on results we prove here and methods from the works above, Kobayashi hyperbolicity has also been extended to effectively parametrized families of minimal manifolds of general type by 
 Deng \cite{Deng}.

To go beyond the canonically polarized case, in this paper we 
take a different path based on the approach of Viehweg and Zuo, 
where the key first step is to refine the Hodge theoretic constructions of~\cite{Vie-Zuo03a} 
(and subsequently~\cite{PS15}), with the ultimate goal of  ``generically" endowing any complex line $\mathbb C$ in $V$ 
with a metric with sufficiently negative curvature;  
this is the content of \S\ref{sect:HodgeModules}. The 
next step, presented in \S\ref{sect:Hypo}, is to extend this metric to a singular metric on $\mathbb C$ 
whose curvature current violates the singular Ahlfors-Schwarz inequality. A review of the line of work that has inspired
this approach to hyperbolicity can be found at the end of \cite[\S1]{Vie-Zuo03a}.

\subsection{Two-dimensional parameter spaces in the general case}
As mentioned at the outset, the results in Theorem~\ref{thm:BHMain} and Theorem~\ref{thm:MaxVar} are expected to 
hold for families of manifolds of lower Kodaira dimension as well,  assuming that they have semiample 
canonical bundle or, more generally, admit a good minimal model (this last condition also includes the case of arbitrary fibers of general type).

On a related note, in~\cite[Thm.~A]{PS15} it is shown that the base $V$ of any smooth family 
whose geometric generic fiber admits a good minimal model, and which has maximal variation, is of log general 
type. Thus the Green-Griffiths-Lang conjecture again predicts hyperbolicity properties for $V$.
Note that when $\dim V = 1$, the two properties are equivalent, and had already been established in \cite{Vie-Zuo01}.
We finish the paper by establishing such results in the case when $V$ is  two-dimensional.

\begin{theorem}\label{thm:MainGG}
Let $f_U\colon U \to V$ be a smooth family of projective manifolds, with maximal variation.
Assume that $V$ is a quasi-projective surface.

\begin{enumerate}
\item \label{item:GG1} If the geometric generic fiber of $f$ has a good minimal model,
then every entire curve $\gamma\colon\CC \to V$ is algebraically degenerate. 
\item \label{item:GG2} Moreover, if the fibers are of general type, then the exceptional locus 
$\Exc(V)$ is a proper subset of $V$.
\end{enumerate}
\end{theorem}

As a consequence of Theorem~\ref{thm:MainGG}, we can extend 
Theorem~\ref{thm:BHMain} to the case of moduli of polarized manifolds,  not necessarily of general type, 
as long as $V$ is two-dimensional.

\begin{corollary}
\label{cor:MainBH}
Let $V$ be a quasi-projective surface admitting a morphism 
$\sigma\colon V \to P_h$ induced by a smooth family $f_U\colon U \to V$ in $\sP_h (V)$.
If $\sigma$ is quasi-finite, then $V$ is Brody hyperbolic.
\end{corollary}

\subsection{Outline of the argument}
The paper follows the beautiful strategy towards proving hyperbolicity for parameter spaces that was developed
in the series of works of Viehweg-Zuo \cite{Vie-Zuo01},  \cite{VZ02}, \cite{Vie-Zuo03a}. 
It relies also on the extension to Hodge modules provided in \cite{PS15} of some Hodge-theoretic constructions in these papers, which in turns enables the level of generality we consider.
Here are the key steps; in each of them we describe what is the new input needed in order to go beyond the canonically polarized case in \cite{Vie-Zuo03a}.

(1) First, one constructs a special Hodge theoretic object on a compactification $Y$ of (a birational model of) the base $V$, namely a graded subsheaf 
$(\sF_\bullet, \theta_\bullet)$ of a Higgs bundle $(\sE_\bullet, \theta_\bullet)$ associated to a Deligne canonical extension of a 
variation of Hodge structure (VHS) supported outside of a simple normal crossing divisor $D + S$, where $D = Y \smallsetminus V$. The system $\sF_\bullet$ encodes the data of maximal variation and has positivity properties due to general Hodge theory. 

A large part of the construction follows ideas from \cite{PS15}; a key ingredient is the use of Hodge module extensions of VHS, necessary especially 
when the fibers are no longer assumed to have semiample canonical bundle. A detailed discussion of the construction can be found in \cite[Introduction and \S2]{PS15} (see also \cite{Popa} for an overview).\footnote{We also take the opportunity in the Appendix
to write down a reduction step to the simple normal crossings case; this was stated in \cite{Vie-Zuo03a} and \cite{PS15} in the respective settings,  but the concrete details were not included.}
However, we make some modifications that lead to an a priori slightly different Higgs sheaf 
$(\sF_\bullet, \theta_\bullet)$; the reason is that we crucially need the induced map $\sT_Y \to \sF_0^{\vee} \otimes \sF_1$ to  
coincide  generically with the Kodaira-Spencer map of the original family. This can be accomplished when the fibers are minimal of general type by appealing to a vanishing theorem due to Bogomolov and Sommese. (The construction for canonically polarized fibers in  \cite{Vie-Zuo03a} appeals to Kodaira-Nakano vanishing, which may fail to hold in this context.)
We note that this is the only point in the paper where it is necessary to work with minimal varieties of general type, 
and which needs to be overcome in order to answer the Viehweg-Zuo question in the arbitrary polarized case.

Given a holomorphic map $\gamma \colon \CC \to V$, 
this construction eventually allows us to produce, for each $m \ge 0$,  morphisms 
$$\tau_m \colon \sT_{\CC}^{\otimes m} \longrightarrow \gamma^* (\sL^{-1} \otimes \sE_m),$$
where $\sL$ is a big and nef line bundle on $Y$, positive on $V$, and $\sE_\bullet$ is the Higgs bundle mentioned above. 
This is all done in \S\ref{sect:HodgeModules}.

(2) For the next step, in the case of Viehweg's hyperbolicity conjecture the point was to apply a powerful
criterion detecting the log general type property, due to Campana-P\u aun \cite{CP16}. In the present case of Brody hyperbolicity, 
this step is by contrast of an analytic, and in some sense more elementary flavor. 
Using the relationship with the Kodaira-Spencer map mentioned above, one shows that for some
$m \ge 1$ the map $\tau_m$ factors through 
$$\tau_m \colon \sT_{\CC}^{\otimes m} \longrightarrow \gamma^* \sL^{-1} \otimes \sN_{(\gamma, m)},$$
where $ \sN_{(\gamma, m)}$ is defined as the kernel of the generalized Kodaira-Spencer map
$$\gamma^*\sE_m \longrightarrow \gamma^* \sE_{m+1} \otimes \Omega_{\CC}^1 (P),$$
with $P = \gamma^{-1} (S)$.  As in \cite{Vie-Zuo03a}, 
we use this, together with results about the curvature of Hodge metrics, in order to construct  a sufficiently negative singular 
metric on $\CC$ which violates the Ahlfors-Schwarz inequality. 

The relaxation of the assumption on the fibers of the family again creates technical difficulties compared to the situation in 
\cite{Vie-Zuo03a}, where one could work with Hodge theoretic objects with finite monodromy around the components of $S$. We consider
instead a further perturbation along $S$, which allows us to construct the singular metric we need using only the well-known growth estimates for Hodge metrics at the boundary given in \cite{Sch73} and \cite{CKS}. This does not require any further knowledge about the monodromy, 
and so has the advantage of giving a slightly simplified argument, in a more general situation.
All of this is discussed in  \S\ref{subsect:AS}--\S\ref{subsect:Brody}.

(3) When the base $V$ of the family is a surface, one does not need to appeal to the connection with the Kodaira-Spencer map mentioned in (1). Consequently the requirement that the fibers be minimal of general type, or even have semiample canonical bundle, can be dropped, noting however that for the Hodge theoretic constructions we now necessarily have to use 
the more abstract Hodge module version. Instead, we follow a different approach by using the map $\tau_1$ in order to produce a foliation on $V$ such that  $\gamma (\CC)$ is contained in one of its leaves. Given that by \cite{PS15} we know that $V$ is of log general type, we can then appeal to a result of McQuillan \cite{McQ98}  on the degeneracy of such entire curves, and to an extension to the logarithmic case in \cite{El03}, in order to obtain a contradiction. This is the subject of \S\ref{subsect:GGLocus}.

\subsection{Acknowledgements} 
We thank  Laura DeMarco, Ya Deng, Henri Guenancia, S\'andor Kov\'acs, Mihai P\u aun, Erwan Rousseau, Christian Schnell  and Sai-Kee Yeung for answering our questions and for useful suggestions.
The authors owe a special thanks to Ariyan Javanpeykar, Ruiran Sun and Kang Zuo
for pointing out a mistake in an earlier version of this paper.

\section{Hodge-theoretic constructions}
\label{sect:HodgeModules}

\subsection{Relative (graded) Higgs sheaves}
\label{subsect:Higgs}
We start with a brief discussion of Higgs sheaves with logarithmic poles. We consider the relative setting, which will be necessary for technical 
reasons later on, though most of the time the constructions are needed in the absolute setting.
Suppose $X$ and $Y$ are smooth quasi-projective varieties,  and $f \colon X\to Y$ is a smooth morphism of relative dimension $d$, with $D$ a reduced relative normal crossing divisor over $Y$. 

Recall that an $f$-relative graded Higgs sheaf with log poles along $D$ is a pair $(\sE_\bullet,\theta_\bullet)$ such that 
\begin{enumerate}
\item{$\sE_\bullet$ is a $\ZZ$-graded $\sO_X$-module, with grading bounded from below.}
\item{$\theta_\bullet$ is a grading-preserving $\sO_X$-linear morphism 
\[\theta_\bullet: \sE_\bullet\rarr \Omega^1_{X/Y}(\log D)\otimes\sE_{\bullet+1}\]
satisfying $\theta_\bullet\wedge\theta_\bullet=0$}, where $\Omega^1_{X/Y}(\log D)$ is the sheaf of relative 1-forms with logarithmic poles along $D$; it is called the Higgs field of the sheaf.
\end{enumerate}
A (relative) Higgs sheaf is called a (relative) Higgs bundle if it consists of $\sO_X$-modules that are locally free of finite rank. If $f$ is trivial, then we get the usual notions of a Higgs sheaf or Higgs bundle. The standard example  is the Hodge bundle associated to a variation of Hodge structure (VHS). More generally, for
a VHS $\V$ on $Y\smallsetminus D$, with quasi-unipotent monodromy along the components of $D$, the Deligne extension of the VHS across $D$ with eigenvalues in $[0, 1)$ is a logarithmic VHS, i.e. the extension of the flat bundle is locally free with a flat logarithmic connection, and the extension of the filtration is a filtration by subbundles; see \cite[Prop.~I.5.4]{Deligne70} and \cite[(3.10.5)]{SaiMHM}; see also \cite[2.5]{Kol86}. Hence its generalized Hodge bundle  
$(\sE_\bullet , \theta_{\bullet})$ is a logarithmic Higgs bundle.

We denote by $\shTA_{X/Y}(-\log D)$ the sheaf of relative vector fields with logarithmic zeros along $D$, and consider  its symmetric algebra 
$$\shA_{X/Y}(-\log D) : =\Sym\shTA_{X/Y}(-\log D)$$ 
(or $\shA_{X/Y}^\bullet(-\log D)$ if we want to emphasize its grading). When $D=0$ and $f$ is trivial, we have 
$\shA_X=\grf\shD_X$, where $\shD_X$ is the sheaf of holomorphic differential operators with the order filtration. We have inclusions of graded $\sO_X$-algebras
\[\shA_{X/Y}(-\log D)\hookrightarrow\shA_{X/Y}\hookrightarrow\shA_{X}
\,\,\,\, {\rm and} \,\,\,\,
\shA_{X/Y}(-\log D)\hookrightarrow\shA_{X}(-\log D)\hookrightarrow\shA_{X}.\]
We will consider graded modules over these sheaves of rings. For instance, the associated graded of a filtered $\shD_X$-module (resp. of a filtered vector bundle with  flat connection with log poles along $D$) is an $\shA_{X}^\bullet$ (resp. $\shA_{X}^\bullet (-\log D)$)-module.  The following reinterpretation of the definitions allows us to use relative Higgs sheafs and graded $\shA_{X/Y}(-\log D)$-modules interchangeably.  

\begin{lemma}
The data of a relative Higgs sheaf $(\sE_\bullet,\theta_\bullet)$ with log poles along $D$ is equivalent to that of a graded
 $\shA_{X/Y}^\bullet (-\log D)$-module structure on $\sE_\bullet$, extending the $\sO_X$-module structure. 
 \end{lemma}

The Higgs field $\theta_\bullet$ induces a complex of graded $\sO_X$-modules, de Rham complex
\[
{\rm DR}_{X/Y}^D  (\sE_\bullet) := \big[\sE_\bullet\to \Omega^1_{X/Y}(\log D)\otimes\sE_{\bullet+1}\to \cdots\to  \Omega^d_{X/Y}(\log D)\otimes\sE_{\bullet+d}\big]
\]
and we have
\[{\rm DR}_{X/Y}^D(\sE_\bullet) \simeq {\rm DR}_{X/Y}^D\big(\shA_{X/Y}^\bullet(-\log D)\big)\otimes_{\shA_{X/Y}^\bullet(-\log D)}\sE_\bullet.\]
\begin{definition}[Pull-back of Higgs bundles] \label{pullbackhb}
Let $\sE_\bullet$ be a relative Higgs bundle on $X$, and $\gamma\colon B \to X$  a holomorphic map from a 
complex manifold $B$, such that the support 
$E$ of $\gamma^{-1}(D)$ is relative normal crossing over $Y$ with respect to the induced map $B\to Y$. 
Then the natural $\sO_X$-linear morphism $\sT_{B/Y}(-\log E)\to \gamma^*\sT_{X/Y}(-\log D)$ induces a morphism 
\[\sA_{B/Y}(-\log E)\to \gamma^*\sA_{X/Y}(-\log D)\]
of graded $\sO_B$-algebras. Therefore, $\gamma^*\sE_\bullet$ is a graded $\sA_B(-\log E)$-module, and in particular a relative Higgs bundle on $B$ with Higgs field induced by that of $\sE_\bullet$.
\end{definition}

\subsection{Hodge modules for rank $1$ unitary representations on quasi-projective varieties}
\label{subsect:FilterL}
We discuss Hodge modules for rank $1$ unitary representations,  needed in what follows. 
We fix a line bundle $\sB$ on a smooth quasi-projective variety $X$, and assume that 
$$
\sB^{m} \simeq\sO_X(E),
$$
for some $m\in \N$ and an effective divisor $E=\sum_i a_i D_i$ with simple normal crossing support. 
We denote $D=E_\text{red}$. 
It is well known that, for every $0< i <m$ and every divisor $E'$ supported on $D$, 
the line bundle $\sB^{-i}(E')$ admits a flat connection with logarithmic poles along $D$.
As in \cite[\S3]{EV92}, we set
$$
\sB^{(-i)} =\sB^{-i} \; \big(\sum_i\lfloor\frac{a_i}{m}\rfloor \cdot D_i\big)
$$
the Deligne canonical extension of $\sB^{-i}|_{X\smallsetminus D}$, which is a flat unitary line bundle on $X\smallsetminus D$ coming from a unitary representation of the fundamental group. We also use the notation 
$$\sB^{-i} (*D) = \bigcup_{k \ge 0} \sB^{-i} (kD)$$
for the sheaf of sections of $\sB^{-i}$ with poles of arbitrary order along $D$.
We define  filtrations on $\sB^{(-i)}$, $\sB^{(-i)}(D)$ and $\sB^{-i} (*D)$ by:

\[F_p\sB^{(-i)}(C)=
\begin{cases}
0& \text{if }p<0\\
\sB^{(-i)}(C)& \text{if } p\ge 0,
\end{cases}\]
where $C$ is either $0$ or $D$, and
\begin{equation}\label{fl}
F_p\sB^{-i} (*D)=
\begin{cases}
0& \text{if } p<0\\
\sB^{(-i)}\big((p+1)D\big)& \text{if }p\ge 0.
\end{cases}
\end{equation}
With these filtrations, $\sB^{(-i)}(D)$ is a filtered line bundle with a flat connection with log poles along $D$, and 
$\sB^{-i}(*D)$ is a filtered $\shD_X$-module. Note that in particular we will always consider $\shO_X$ with the trivial filtration $F_k \shO_X = \shO_X$ for $k \ge 0$, and $0$ otherwise, so that $\gr^F_\bullet \sO_X \simeq \sO_X$.

By \cite[(3.10.3) and (3.10.8)]{SaiMHM}, we know that $(\sB^{-i}(*D),\fddot)$ is a direct summand of the filtered $\shD_X$-module underlying $\pi_* {\Q^H_Z}[\dim Z]$, the direct image of the trivial Hodge module on $Z$, where $\pi\colon Z\to X$ is the $m$-th cyclic cover branched along the divisor $E$. 

Note that $\mathrm{gr}^F_{\bullet}\sB^{(-i)}(D)$ is a graded $\shA_{X}^\bullet (-\log D)$-module, while $\mathrm{gr}^F_{\bullet}\sB^{-i}(*D)$ is a graded $\shA_{X}^\bullet$-module. Moreover, the natural inclusions
\[\mathrm{gr}^F_{\bullet}\sB^{(-i)}\hookrightarrow\mathrm{gr}^F_{\bullet}\big(\sB^{(-i)}(D)\big)\hookrightarrow \mathrm{gr}^F_{\bullet}\sB^{-i}(*D)\]
preserve the Higgs structure. We have the following comparison result:

\begin{proposition}
\label{prop:relcomp}
Assume that $f\colon X \to Y$ is a smooth projective morphism  of relative dimension $d$ between smooth quasi-projective varieties, and $D$ is a divisor on $X$ which is relatively normal crossing over $Y$. Then the natural morphism 
$${\rm DR}_{X/Y}^D \big(\mathrm{gr}^F_{\bullet}\sB^{(-i)}\big)\  \longrightarrow {\rm DR}_{X/Y}\big(\mathrm{gr}^F_{\bullet}\sB^{(-i)}(*D)\big)$$
is a quasi-isomorphism of complexes of graded $\sO_X$-modules.
\end{proposition}

\begin{proof}
The absolute case was proved in \cite[\S3.b]{SaiMHM} in a more general setting. The relative case is similar; we sketch
the proof for completeness. 

We define graded sheaves $\sC_{\bullet}$ and $\sN_{\bullet}$ by
\[\sC_\bullet : =\shA_{X/Y}^\bullet\otimes_{\gr^F_\bullet \sO_X} \grf \big(\sB^{(-i)}(D)\big)\]
and 
\[\sN_\bullet : =\shA_{X/Y}^\bullet\otimes_{\shA^\bullet_{X/Y}(- \log D)} \grf \big(\sB^{(-i)}(D)\big).\]
By definition $\sC_\bullet$ is a graded $\big(\shA_{X/Y}^\bullet$-$\shA_{X/Y}^\bullet(- \log D)\big)$-bimodule.
(The $\shA_{X/Y}(- \log D)$-module 
structure is induced by the product rule; that is, locally $x_i\partial_{x_i}\cdot(\nu\otimes l)=x_i\partial_{x_i}\cdot\nu\otimes l-\nu\otimes x_i\partial_{x_i}\cdot l$, if $\nu\otimes l$ is a section of $\sC_\bullet$.)
Assume now that $\sT_{X/Y}(- \log D)$ is freely generated locally by 
$$
\partial_{x_1},\dots, \partial_{x_i}, x_{i+1}\partial_{x_{i+1}},\dots, x_d\partial_{x_d}.
$$ 
The sequence of actions of these elements on $\sC_\bullet$ (via the $\shA_{X/Y}(- \log D)$-module structure described above) gives rise to a Koszul-type complex. Written in a coordinate free way,  this is a complex of $\shA_{X/Y}^\bullet$-modules
\[
{\cB}^\bullet_\bullet= \big[\sC_{\bullet-d}\otimes  \bigwedge^{d}\sT_{X/Y}(- \log D)\to \sC_{\bullet-d+1}\otimes \bigwedge^{d-1}\sT_{X/Y}(- \log D)\to \ldots\to \sC_{\bullet}\big].
\]
Using the fact that $\grf \big(\sL^{(-i)}(D)\big)$ is locally free of rank $1$ over 
$\gr^F_\bullet \sO_X$, one can check that this sequence is regular; therefore, the natural morphism 
\[{\cB}^\bullet_\bullet\rarr H^0 {\cB} ^\bullet_\bullet = \frac{\sC_\bullet}{\sum_{j =1}^i \partial_{x_j} \sC_\bullet + 
\sum_{j = i+1}^d x_j \partial_{x_j} \sC_\bullet} \simeq \sN_\bullet\]
is a quasi-isomorphism of complexes of graded $\shA_{X/Y}^\bullet$-modules. The exactness of the de Rham functor implies that the induced morphism
\[{\rm DR}_{X/Y}({\cB}^\bullet_\bullet)\rarr {\rm DR}_{X/Y}(\sN_\bullet)\]
is a quasi-isomorphism as well.
Moreover,  one also sees that the natural morphism 
\[{\rm DR}_{X/Y}\big(\sC_{\bullet-d+p}\otimes \bigwedge^{d-p}\sT_{X/Y}(-\log D)\big)
\rarr \gr^F_{\bullet+p}\sB^{(-i)}\otimes \Omega^{p}_{X/Y}(\log~D)\]
is a quasi-isomorphism, thanks to the natural isomorphism given by contraction
\[\omega_{X/Y}(D)\otimes\bigwedge^{d-p}\sT_{X/Y}(-\log D)\simeq\Omega^{p}_{X/Y}(\log~D),\] 
and the fact that ${\rm DR}_{X/Y} \big( \sA_{X/Y}^\bullet \big)$ is quasi-isomorphic to $\omega_{X/Y}$.
Therefore, we find that  
${\rm DR}_{X/Y}( \cB^\bullet_\bullet)$ and ${\rm DR}_{X/Y}^D(\mathrm{gr}^F_{\bullet}\sB^{(-i)})$
are quasi-isomorphic.
We now conclude by noting that there is an isomorphism of $\shA_{X/Y}^\bullet$-modules
$$
\sN_\bullet\simeq \grf \sL^{-i} (*D);
$$
see for instance \cite[Prop. 4.2.18]{Bjork} (where it is stated locally, for more general $\Dmod$-modules).
 \end{proof}

\subsection{Hodge modules and branched coverings} 
\label{subsect:Hodge}
This section is essentially a review of the constructions in \cite[\S2.3 and 2.4]{PS15}, but with a twist which is important for 
the applications in this paper.  We assume that we have a morphism of smooth projective varieties $f\colon X\to Y$, with connected fibers, and 
with $\dim Y = n$ and $\dim X = n + d$. Let $\sA$ be a line bundle on $Y$,  and define
$$
\sB: =\omega_{X/Y}\otimes f^*\sA^{-1}.
$$
We make the following assumption:
\begin{equation}\label{gsection} 
\textup{There exists $0 \neq s \in H^0 (X, \sB^m)$  for some $m>0$.}
\end{equation}

The section $s$ defines a branched cover $\psi\colon X_m\to X$ of degree $m$. 
Let $\delta\colon Z\rarr X_m$ be a desingularization of the normalization of $X_m$, which is irreducible if $m$ is chosen to be 
minimal,  and set $\pi=  \psi\circ \delta$ and $h=f\circ \pi$, as in the diagram

$$
  
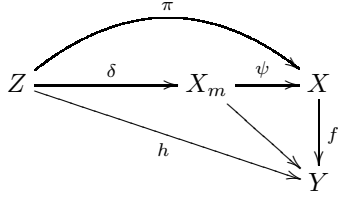
\xymatrix{
Z  \ar@/^9mm/[rrr]^{\pi}  \ar[rr]^{\delta}   \ar[drrr]_{h} &&   X_m  
 \ar[r]^{\psi}  \ar[dr] &   X\ar[d]^{f} \\
&&& Y
}
$$

Let $\shA_Y = \text{Sym } \sT_Y$, with the natural grading, and similarly for $\shA_X$. 
A morphism of graded $\shA_Y$-modules
\begin{equation}\label{eq:theMor1}
\bR f_*(\omega_{X/Y}\otimes_{\sO_X} \sB^{-1} \overset{\bf L}\otimes_{\shA_X}f^*{\shA_Y})\rarr \bR h_*(\omega_{Z/Y}\overset{\bf L}\otimes_{\shA_Z}h^*\shA_Y).
\end{equation}
is constructed in \cite[\S2.4]{PS15}. (We use the notation $\sB^{-1}  \overset{\bf L}\otimes_{\shA_X}f^*{\shA_Y}$ 
as shorthand for $\sB^{-1} \otimes_{\sO_X}  \gr^F_\bullet \sO_X \overset{\bf L}\otimes_{\shA_X}f^*{\shA_Y}$, where we make use of the $\sA_X$-module structure on $\gr^F_\bullet \sO_X \simeq \sO_X$.)
Here we will construct a similar but slightly different morphism, that a priori coincides with the one 
in ($\ref{eq:theMor1}$) only generically.  
The reason for this different construction will become 
clear in \S\ref{subsect:KS}, where we need to compare Hodge sheaves constructed out of 
branched coverings with others that are naturally related to Kodaira-Spencer maps. 

Before starting the construction, recall from \cite[\S3.b]{SaiMHM} that $\sB$ uniquely determines a filtered $\shD_X$-module 
$(\sB^{-1}_*, \fddot)$ with strict support $X$, which extends $(\sB^{-1}|_{X\smallsetminus {\rm div}(s)},\fddot)$,  where the filtration on the latter is the trivial filtration; notice that the filtered $\shD_X$-module is exactly $(\sB^{-1}(*D), \fddot)$, when $D={\rm div}(s)_\text{red}$ is normal crossing.  Moreover, 
$(\sB^{-1}_*, \fddot)$ is a direct summand of the  filtered $\sD_X$-module $\cH^0\pi_+ (\sO_Z, F_\bullet)$.

\begin{lemma} \label{lem:inclusion}
We have a natural inclusion
\[\sB^{-1}\hookrightarrow F_0 \; \sB^{-1}_*.\]
\end{lemma}
\begin{proof}
Let $\mu \colon X'\to X$ be a log resolution of the divisor div$(s)$ which is an isomorphism on its complement. 
Define $D' =(\mu^*\text{div}(s))_\text{red}$ 
and $\sB'=\mu^*\sB$. Then, according to the discussion in \S\ref{subsect:FilterL},  
$(\sB'^{-1} (*D'),\fddot)$ defined as in \eqref{fl} is a direct summand of a 
$\shD_{X'}$-module underlying a Hodge module. By the strictness of the direct image functor for Hodge modules, we have 
\[\mu_+(\sB'^{-1} (*D'),\fddot)=(\sB^{-1}_*, \fddot)\]
and 
\[\mu_*F_0\sB'^{-1} (*D')  = F_0\sB^{-1}_*.\]
On the other hand, by construction we have the injection
\[\sB'^{-1}\subset F_0\sB'^{-1} (*D'),\]
and so the statement follows from the projection formula. 
\end{proof}

We now proceed with our construction.
The inclusion in Lemma~\ref{lem:inclusion} induces a morphism of graded $\shA_X$-modules
\[ \sB^{-1} \rarr \grf \sB^{-1}_*,\]
with the trivial graded $\shA_X$-module structure on $\sB^{-1}$.
This in turn induces a morphism
\[\bR f_*(\omega_{X/Y}\otimes_{\sO_X} \sB^{-1}\overset{\bf L}\otimes_{\shA_X}f^*{\shA_Y})\rarr \bR f_*(\omega_{X/Y}\otimes_{\shO_X} \grf \sB^{-1}_*\overset{\bf L}\otimes_{\shA_X}f^*{\shA_Y}).\]
Now the right hand side is a direct summand of the object 
$\bR h_*(\omega_{Z/Y}\overset{\bf L}\otimes_{\shA_Z}h^*\shA_Y)$; indeed, using \cite[Theorem~2.9]{PS-GV}, we have an isomorphism
$$\gr_\bullet^F h_+ (\shO_Z, F_\bullet) \simeq  \bR h_*(\omega_{Z/Y}\overset{\bf L}\otimes_{\shA_Z}h^*\shA_Y),$$
and we combine this with the filtered direct summand inclusion of $(\sB^{-1}_*, \fddot)$ in $\cH^0\pi_+ (\sO_Z, F_\bullet)$. 
Therefore we get an induced morphism
\begin{equation}\label{eq:theMor}
\bR f_*(\omega_{X/Y}\otimes_{\sO_X} \sB^{-1}\overset{\bf L}\otimes_{\shA_X}f^*{\shA_Y})\rarr\bR h_*(\omega_{Z/Y}\overset{\bf L}\otimes_{\shA_Z}h^*\shA_Y),
\end{equation}
which factors through $\bR f_*(\omega_{X/Y}\otimes\grf \sB^{-1}_*\overset{\bf L}\otimes_{\shA_X}f^*{\shA_Y})$. One can check that the morphisms \eqref{eq:theMor1} and \eqref{eq:theMor} coincide over the locus where $h$ is smooth; they are however not necessarily the same globally. 

\medskip

Let now $(\cM, F_\bullet)$ be the filtered $\sD_Y$-module underlying the Tate twist $M(d)$ of the pure polarizable Hodge module $M$ which is the direct summand of 
$\cH^0h_* {\bQ_Z^H}[n+d]$ strictly supported on $Y$. By \cite[Prop.~2.4]{PS15},  we then have that gr$^F_{\bullet} \cM$ is a direct summand of $R^0 h_*(\omega_{Z/Y}\overset{\bf L}\otimes_{\shA_Z}h^*\shA_Y)$.

\begin{definition}
We define a graded $\shA_Y$-module $\shG_\bullet$ as the image of the composition
\[
R^0 f_*(\omega_{X/Y}\otimes_{\sO_X} \sB^{-1}\overset{\bf L}\otimes_{\shA_X}f^*{\shA_Y})\to R^0 h_*(\omega_{Z/Y}\overset{\bf L}\otimes_{\shA_Z}h^*\shA_Y)\to \gr^F_{\bullet} \cM,
\]
where the second morphism is given by projection.
\end{definition}

Recall that $D_f$ denotes the singular locus of $f$. 
We gather the constructions above and further properties in the following result, which is essentially \cite[Thm.~2.2]{PS15}; although as pointed out above the new morphism \eqref{eq:theMor} is constructed slightly differently,  the proof is identical.

\begin{theorem}\label{thm:G-sheaf}
With the above notation, assuming \eqref{gsection}, the coherent graded $\shA_Y$-module $\GF$ satisfies the following properties:
\begin{enumerate}
\item \label{G0}
There is an isomorphism $\shG_0 \simeq \sA$.
\item 
Each $\shG_k$ is torsion-free on $X\smallsetminus D_f$.
\item 
There is an inclusion of graded $\shA_Y$-modules $\GF\subseteq\grf\cM$. 
\end{enumerate}
\end{theorem}

\subsection{Basic set-up}
\label{subsect:summary}

We consider a smooth family $f_U \colon U \to V$ of projective varieties, whose geometric generic fiber admits a good minimal model. (This includes for instance families of varieties of general type, or of varieties whose canonical bundle is semiample.)
We assume that the family has maximal variation; following the strategy in \cite{Vie-Zuo03a}, together with the technical extensions in \cite{PS15}, our aim in the next two sections is to endow entire curves inside (a birational model of) $V$ with Hodge theoretic objects that will be later used in order to conclude hyperbolicity.

In order to accomplish this, we will use a technical statement about the existence of sections (or the generic global generation) 
for suitable line bundles on a
modification of the family $f_U$. This is proved in the Appendix in \S\ref{sect:appendix}, in Propositions \ref{thm:ggrefine} and \ref{section_SNC}. The idea and most of the details can be found in \cite{Vie-Zuo03a}; for a detailed discussion please see the Appendix.

\subsection{Main construction on $\mathbb C$.}
In the set-up of \S\ref{subsect:summary}, our aim here is to use the  constructions in the previous sections in order to produce
interesting Hodge-theoretic sheaves on $\CC$, assuming the existence of a holomorphic mapping $\gamma \colon \CC \to V$. 

\noindent
{\bf Assumption}: all VHS appearing in this paper are assumed to be polarizable, and all local monodromies to be quasi-unipotent; see for instance \cite{Sch73} for the definitions. This is of course the case for any geometric VHS, i.e. the Gauss-Manin connection of a smooth family of projective manifolds, thanks to the monodromy theorem (see for instance \cite[Lem.~4.5]{Sch73}). In general, fixing a polarization induces the Hodge metric on the associated Higgs bundle, its singularities at the boundary will play a crucial role in \S\ref{subsect:AS}. 

\medskip

We start with the key output of the Hodge theoretic constructions above, following arguments in \cite{PS15}. 
According to the strategy in \cite{Vie-Zuo03a}, it will later be combined with analytic arguments in order to conclude the non-existence of dense entire curves.

\begin{proposition}\label{prop:summary}
Let $f_U\colon U\to V$ be a smooth family of projective varieties, with maximal variation, 
and whose geometric generic fiber has a  good minimal model. 
Then, after possibly replacing $V$ by a birational model,
there exists a smooth projective compactification $Y$ of $V$, with $D = Y \smallsetminus  V$ a 
simple normal crossing divisor, together with a big and nef line bundle $\sL$ and  an inclusion of graded $\sA_Y^\bullet(-\log D)$-modules
$$ 
(\sF _{\bullet}, \theta_{\bullet}) \subseteq (\sE_{\bullet}, \theta_{\bullet}),
$$
on $Y$, that verify the following properties:

\begin{enumerate}
\item $(\sE_\bullet, \theta_\bullet)$ is the Higgs bundle underlying the Deligne extension with eigenvalues in  $[0, 1)$
of a VHS defined outside of a simple normal crossing divisor $D + S$. 
\item \label{item:ample} $\sF_0$ is a line bundle, and we have an inclusion $\sL\subseteq \sF_0$ which is an isomorphism 
on $V$. 
\item \label{item:2} If $\gamma\colon \bbC\to  V \subseteq Y$ is a holomorphic map, then for each $k\ge 0$ there exists a morphism
$$
\tau_{(\gamma, k)} : \sT_{\mathbb C}^{\otimes k}  \to \gamma^*\big(\bigotimes^{k}  \sT_Y  (- {\rm log} D)\big)
   \to  \gamma^*( {\sF_0}^{-1}  \otimes \sE_{k})  \rightarrow \gamma^*(\sL^{-1} \otimes \sE_{k}).
$$

\end{enumerate}
\end{proposition}

\begin{proof}
We consider $f \colon X \to Y$ as in Proposition~\ref{thm:ggrefine} in the Appendix.\footnote{Unlike in the Appendix, here we denote the 
original family $U \to V$, and we keep this notation after passing to a birational model, since there is no danger of confusion.} Thus  there exist an integer $m > 0$ and a line bundle $\sA$ on $Y$, of the form $\sA= \sL(D_Y)$ with $\sL$ ample and $D_Y \ge D$, such that 
$$H^0\big(X, (\omega_{X/Y} \otimes {f}^*\sA^{-1})^m\big)\neq 0.$$
This means that we can apply the constructions in \S\ref{subsect:Hodge}; we set
$$\sB =\omega_{X/Y } \otimes {f}^*\sA^{-1}$$ 
and pick $0 \neq s \in H^0(X, \sB^m)$. 
Associated to this section, by applying Theorem \ref{thm:G-sheaf}, we obtain a Hodge sheaf $\GF$ and a Hodge module $M$ on $Y$. 

For the next construction, we would like to assume that there is an effective divisor $S$ on $Y$ such that 
the singular locus of $M$ is (contained in) $D + S$, and that $D+ S$ has simple normal crossings. In fact, and this is sufficient, we can only accomplish this outside of 
a closed subset of codimension at least $2$, as follows. We consider a further birational model $\tilde f\colon \tilde X \to \tilde Y$ 
as in Proposition \ref{section_SNC}, imposing that the singular locus $S$ of $M$ contain the branch locus $\Delta_\tau$ in that statement. 
Using the notation $\mu \colon \tilde Y \to Y$ for the birational map on the base, we obtain that there exists a closed subset $T$ in $\tilde Y$, of codimension at least $2$, such that  $s$ induces a new section 
\[\tilde s\in H^0\big(\tilde X_0, (\omega_{\tilde X/\tilde Y}\otimes {\tilde f}^*\mu^*\sA^{-1})^{m}\big),\] 
with $\tilde Y_0=\tilde Y\smallsetminus T$ and $\tilde X_0 = \tilde f^{-1} (\tilde Y_0)$, such that 
$\tilde s$ and $s$ coincide on the fibers over points away from $S$, where $\mu$ is the identity map.

We again record the conclusion of Theorem  \ref{thm:G-sheaf} for the new family, over $\tilde Y_0$ only. Note that since the sections coincide away from $S$, the new pure Hodge module is the unique extension with strict support of the same VHS as $M$, defined on the complement of $D+ S$. We can therefore revert to the original notation, and assume that we have a Hodge module $M$ and 
a Hodge subsheaf $\GF$ on $Y$, such that on an open subset $Y_0$ with complement of codimension at least $2$ they coincide with those constructed
as above from the section $s$, and in addition the divisor $D + S$ (and in particular the singular locus of $M$) has simple normal crossings.
Note that because of the birational modification, $\sL$ is now only a big and nef line bundle.

We now take $(\sE_\bullet, \theta_\bullet)$ to be the Higgs bundle on $Y$ underlying the Deligne extension with eigenvalues in $[0, 1)$
of the VHS that coincides with $M$ outside of $D + S$. 
Following \cite[\S2.7 and \S2.8]{PS15}, on the open set $Y_0$ we define a subsheaf $(\sF_\bullet, \theta_\bullet)$ of $(\sE_\bullet,  \theta_\bullet)$ by
$$
\sF_{\bullet} = \bigl( \sG_{\bullet} \cap \sE_{\bullet}   \bigr)^{\vee\vee}.
$$
Note that the intersection makes sense, since both $\sG_{\bullet}$ and $\sE_\bullet$ are contained in  $\gr^F_\bullet \cM$.
Precisely as in~\cite[Prop.~2.14 and~Prop.~2.15]{PS15}, on $Y_0$ one has the 
following properties for $\sF_{\bullet}$: 

\begin{enumerate}
\item \label{aaa} We have $\sA(-D)\subseteq \sF_0\subseteq \sA$, for some integer $l>0$.
\item \label{bbb} The Higgs field $\theta$ maps $\sF_k$ into $\Omega^1_Y(\log D)\otimes \sF_{k+1}$.
\end{enumerate}

Now since the complement of $Y_0$ has codimension at least $2$, the sheaves $\shF_k$ have a unique reflexive extension to the entire $Y$. As all the other sheaves appearing in them are locally free, the maps in \ref{aaa}
and \ref{bbb} extend uniquely as well, and hence both properties continue to hold on $Y$. This realizes the 
global construction.

Note that $\sF_0$ is a reflexive sheaf of rank $1$ on the smooth variety $Y$, and hence is a line bundle.
Thus~\ref{aaa} shows Item~\ref{item:ample}, while~\ref{bbb} leads to Item~\ref{item:2} by the
 following construction. 
Note that~\ref{bbb} means $\sF_\bullet$ is an $\sA_Y(- \log D)$-module. The $\sA_Y(- \log D)$-module structure induces a map 
\[\rho_k:\bigotimes^{k}  \sT_Y(-\log D)\to \Sym^k\sT_Y(-\log D)\to {\sF_0}^{-1}\otimes \sF_k \to {\sF_0}^{-1} \otimes \sE_k.\]
By composing $\rho_k$ with the $k$-th tensor power of the differential 
\[d\gamma: \sT_\bbC\to \gamma^* \sT_Y (- {\rm log}D),\] 
we obtain 
$$
\tau_{(\gamma, k)}:\sT_{\mathbb C}^{\otimes k}  \xrightarrow{d\gamma^{\otimes k}} \gamma^*\big(\bigotimes^{k}  \sT_Y (- {\rm log}D) \big)\xrightarrow{\gamma^*\rho_k}\gamma^*( {\sF_0}^{-1}  \otimes \sE_{k})  \hookrightarrow \gamma^*(\sL^{-1} \otimes \sE_{k}),
$$
where the last morphism is induced by the inclusion of $\sL$ into $\sF_0$.
\end{proof}

\begin{remark}
If  $f_U\colon U\to V$ has fibers with semiample canonical bundle, then by Proposition \ref{thm:ggrefine} we may also assume that $\sB^m$ is globally generated over $f^{-1} (V)$ in the result above. This will be used in the next section.
\end{remark}

Finally, we record a fact that will be of use later on. 

\begin{lemma}\label{nonzero}
In the notation of Proposition \ref{prop:summary}, the Higgs map
$$\theta_0 \colon \sF_0 \longrightarrow \sF_1 \otimes \Omega^1_Y (\log D)$$
is injective.
\end{lemma}
\begin{proof}
It suffices to show that $\theta_0$ is not the zero map, since $\sF_0$ is a line bundle and $\sF_1$ is torsion free.
By Item \ref{item:ample}, we know that $\sF_0$ is a big line bundle. On the other hand, if $\theta_0$ were identically zero, then we would have that $\sF_0 \subseteq K_0$, where 
$$K_0 := {\rm ker} \big(\theta_0 \colon \sE_0 \to \sE_1 \otimes \Omega^1_Y (\log D +S)\big).$$
Now $K_0^\vee$ is a weakly positive sheaf by \cite[Theorem~4.8]{PW16} (an easy consequence of the results of \cite{Zuo00} and 
\cite{Bru15} in the unipotent case), so this would imply that $\sF_0^{-1}$ is also a pseudoeffective line bundle, a contradiction.
\end{proof}

\subsection{Further refinements for families of minimal manifolds of general type}
\label{subsect:KS}
In the current section, assuming that the members of the family 
are minimal and of general type, we will establish a connection between 
the sheaf $(\sF, \theta)$ defined in Proposition~\ref{prop:summary} 
and the Kodaira-Spencer map of $f$. In the canonically polarized case treated in \cite{Vie-Zuo03a}, 
an analogous statement is proved as an application of the Akizuki-Nakano vanishing theorem, which 
in the present context is not available any more; we will be able to achieve this using a different 
argument based on transversality and a more restrictive vanishing theorem due to Bogomolov and Sommese.

We continue to be in the set-up of \S\ref{subsect:summary},  and we fix the morphism $f \colon X \to Y$ as in the proof of
Proposition~\ref{prop:summary}.
We define a new graded $\shA_Y$-module $\widetilde \sF_\bullet$ by 
\begin{equation}\label{eq:FsheafKS}
\widetilde{\FF}= R^0 f_*\big( \omega_{X/Y} \otimes_{\sO_X} \sB^{-1}\otimes \gr^F_\bullet \sO_X \overset{\bf L}{\otimes}_{\shA_X}f^*{\shA_Y}\big),
\end{equation}
i.e. the left hand side of ($\ref{eq:theMor}$), where the $\shA_Y$-module structure is induced by the $f^*{\shA_Y}$-module structure on $\omega_{X/Y} \otimes_{\sO_X} \sB^{-1}\otimes \gr^F_\bullet \sO_X \overset{\bf L}{\otimes}_{\shA_X}f^*{\shA_Y}$.
This structure induces a morphism 
\begin{equation}\label{KS}
\shTA_Y\rarr {\widetilde\shF_0}^{-1} \otimes \widetilde\shF_1.
\end{equation}
Also, by the projection formula, we have $\widetilde\shF_0=\sA$.

On the other hand, over the locus where $f$ is smooth, using the fact that the natural morphism 
\begin{equation}\label{eq:QuasiComp}
[\shA_X^{\bullet-d}\otimes \bigwedge^{d}\sT_{X/Y}\rarr\shA_X^{\bullet-d+1}\otimes \bigwedge^{d-1}\sT_{X/Y}\rarr\cdots\rarr \shA_X^\bullet]\rarr f^*\shA_Y^\bullet
\end{equation}
induced by the natural mapping $\sT_X \to f^* \sT_Y$ is a quasi-isomorphism of complexes of graded $\shA_X$-modules (see for example~\cite[Lem.~14.3.5]{Pha79}\footnote{In \emph{loc. cit.}  it is stated for $\Dmod_X$ and $\Dmod_Y$ respectively, as opposed to their associated graded objects.}), we know that
\[\widetilde\shF_\bullet\simeq R^0 f_*\big(f^*\sA \otimes \grf\sO_X\otimes_{\shA_X}{\rm DR}_{X/Y}{(\shA_X^\bullet)}\big).\]
In particular, over this locus we have
$$\widetilde{\sF}_1 \simeq\sA\otimes R^1 f_*\sT_{X/Y}.$$
Therefore, by construction we obtain:

\begin{lemma}\label{lm:compKS}
Over $V = Y \smallsetminus D$, the morphism \eqref{KS} is precisely the Kodaira--Spencer map 
\[\shTA_Y\rarr R^1 f_*\sT_{X/Y}.\]
\end{lemma}

Consequently, in order to establish a connection between the 
Kodaira-Spencer map and  $(\sF_\bullet, \theta_\bullet)$ in Proposition~\ref{prop:summary}, 
 it suffices to establish one between $(\wtilde \sF_\bullet, \wtilde \theta_\bullet)$ 
and $(\sF_\bullet, \theta_\bullet)$. This follows immediately from the next result.

\begin{proposition}\label{prop:GenericIso}
For $k\le 1$, the natural morphism 
\[\widetilde\shF_k\rarr\shG_k\]
is generically an isomorphism.
\end{proposition}

\begin{proof}
For $k=0$, the statement follows from by Theorem \ref{thm:G-sheaf}, Item \ref{G0}. 
We now focus on the $k=1$ case. By the basepoint-free theorem, the fibers have semiample 
canonical bundle, hence the second part of Proposition~\ref{thm:ggrefine} applies and so $\sB^m$ is 
generated by global sections over $f^{-1}(V)$. 
Replacing $Y$ by $V$, after shrinking it further if necessary, and $X$ by $\mu^{-1}(V)$,  
we can assume that $f|_H \colon H\rarr Y$, $f$ and $h$ are smooth and $\sB^m\simeq\sO_X(H)$ is globally generated, where $H$ is a smooth divisor transversal to the fibers.
Here $h$ is the morphism defined in \S\ref{subsect:Hodge} by the resolution of the branched covering associated to the global section defining $H$.
Since $h$ is smooth, we have $\cH^0h_*{\bQ_Z^H}[n+d]=M(-d)$ and so it is enough to show that
the morphism 
\begin{equation}\label{FGmor}
\widetilde\shF_1\rarr R^0 f_*\big(\omega_{X/Y}\otimes\grf \sB^{-1}_* \overset{\bf L}{\otimes}_{\shA_X}f^*{\shA_Y}\big)_1
\end{equation}
defined in \S\ref{subsect:Hodge} is injective.

On the other hand, as $f$ is smooth, as we have seen above we have 
\[\widetilde\shF_\bullet\simeq R^0 f_*\big(\sB^{-1}\otimes \grf\sO_X\otimes_{\shA_X}{\rm DR}_{X/Y}{(\shA_X^\bullet)}\big)\]
In particular, since $\sB^{-1}=\sB^{(-1)}$ we have
\[\widetilde\shF_1\simeq R^0 f_*\big(\sB^{-1}\otimes \grf\sO_X\otimes_{\shA_X}{\rm DR}_{X/Y}{(\shA_X^\bullet)}\big)_1\simeq R^1 f_*\big(\sB^{-1}\otimes \Omega^{d-1}_{X/Y}\big).\]
Moreover, since $H$ is smooth (so that  $\sB^{-1}_*$ is the same as $\sB^{-1} (*D)$) and transversal to the fibers, 
according to~\eqref{eq:QuasiComp} and Proposition~\ref{prop:relcomp}, we also have
\[R^0 f_*\big(\omega_{X/Y}\otimes\grf \sB^{-1}_* \overset{\bf L}{\otimes}_{\shA_X}f^*{\shA_Y}\big)_1\simeq R^1 f_*\big(\sB^{-1}\otimes \Omega^{d-1}_{X/Y}(\log H)\big).\]
It follows that the morphism in \eqref{FGmor} is induced by the first map of the following short exact sequence
\[0\rarr\Omega^{d-1}_{X/Y}\rarr\Omega^{d-1}_{X/Y}(\log H)\rarr \Omega^{d-2}_{H/Y}\rarr 0.\]

Notice that 
\[\sB|_F\simeq \omega_F\]
on each fiber $F$ of $f$. Since $\omega_F$ is big and nef, by calculating the top self-intersection number we see that $\sB|_{F_H}$ is big on the general fiber $F_H$ of $f|_H$ for general $H \in |\sB^m|$.
Then, according to the Bogomolov-Sommese vanishing theorem (see for instance~\cite[Cor.~6.9]{EV92}), 
we know that
$${f|_{H}}_*(\sB^{-1}\otimes\Omega^{d-2}_{H/Y})=0$$
generically, and hence everywhere since it is torsion-free. Therefore, we get the desired injectivity
for the morphism in ($\ref{FGmor}$), and this finishes the proof of the proposition.

\end{proof}

\begin{corollary}\label{cor:GTExtra}
In the situation of Proposition~\ref{prop:summary},
if we further assume that the fibers of $f_U$ are minimal and of general 
type, then the natural morphism induced by the $\sA_{Y}(-\log D)$-module structure
$$\sT_{Y}(-\log D) \longrightarrow {\sF_0}^{-1} \otimes \sF_1$$ 
coincides with the Kodaira-Spencer map of $f$ over a Zariski open subset of $V$. 
\end{corollary}

\begin{proof}
Thanks to Proposition~\ref{prop:GenericIso}, we know that  
the sheaves $\wtilde \sF_{k}$ and $\sG_{k}$ are generically isomorphic for $k=0,1$. On the other hand, $\sF_\bullet$ and $\sG_\bullet$ 
are generically the same by construction. Therefore, $\sF_k$ and $\wtilde \sF_{k}$ are generically isomorphic for $k=0,1$.
But  Lemma~\ref{lm:compKS} says that the morphism $\sT_Y \to \wtilde \sF_0^{-1} \otimes \wtilde \sF_1$ 
coincides with the Kodaira-Spencer map of $f$ over $V$, which proves the claim. 
\end{proof}

\section{Hyperbolicity properties of base spaces of families}
\label{sect:Hypo}

In this final part we establish the two main results of this paper, Theorem~\ref{thm:MaxVar} (and implicitly Theorem~\ref{thm:BHMain}) and 
Theorem~\ref{thm:MainGG}. Besides Proposition~\ref{prop:summary} and Corollary~\ref{cor:GTExtra}, the main ingredient in the 
proofs of these theorems is Proposition~\ref{prop:BHHiggs} below.

\subsection{Preliminaries on singular metrics on line bundles, and on Hodge metrics}
\label{subsect:Sing}

We start with a construction and analysis of  particular singular metrics on line bundles that will be of use later on. This follows very closely the material in \cite[p.136--139]{Vie-Zuo03a}. Nevertheless we include the details for later reference, and we also make a distinction between the boundary divisors $D$ and $S$, as the perturbation along $S$ will later allow us to bypass monodromy arguments in \cite{Vie-Zuo03a} in order to extend the range of applicability.

We note to begin with that a priori by a singular metric on a line bundle $\sL$ we mean, as in \cite[\S13]{HPS}, a metric $h$ given by a weight function 
$e^{-\varphi}$, where $\varphi$ is taken to only be a measurable function with values in $[-\infty, \infty]$. In this way, the notion is compatible with that 
of a singular metric on a vector bundle, in the sense of Berndtsson, P\u aun and Takayama  (cf. \cite[\S17]{HPS}), which will also make an appearance 
later on.  In the line bundle case, usually it is also required that $\varphi$ be locally integrable, in which case one can talk about its curvature 
form as a $(1,1)$-current; for this we use the standard notation   
$$F(\sL, h) = \frac{\sqrt{-1}}{\pi}~\partial\bar\partial\varphi = -\frac{\sqrt{-1}}{2\pi}~\partial\bar\partial\log \|e\|^2_h,$$
where $e$ is a holomorphic section which trivializes $\sL$ locally.

Let $(Y,D+S)$ be a pair consisting of a smooth projective variety $Y$ and simple normal
crossings divisors $D= D_1 + \cdots + D_k$
and $S = S_1 + \cdots + S_\ell$. For $i\in \{ 1, \ldots, k  \}$ and $j\in \{ 1, \ldots, \ell \}$
pick 
$$f_{D_i} \in H^0 \big(Y, \sO_Y(D_i)\big), \,\,\,\,\,\, f_{S_j}\in H^0\big(Y, \sO_Y(S_j)\big)$$
such that $ D_i=(f_{D_i}=0)$ and $S_j=(f_{S_j}=0)$.
For each $i$, $j$, let $g_{D_i}$, $g_{S_j}$ be smooth metrics on $\sO_Y(D_i)$
and $\sO_Y(S_i)$, respectively; after rescaling, we may assume $\|f_{D_i}\|_{g_{D_i}}<1$ and $\|f_{S_j}\|_{g_{S_j}}<1$. 

Now, for each $i$ and $j$, set
$$
r_{D_i}= -\log \|  f_{D_i}   \|^2_{g_{D_i}},     \; \;   \;  \;  r_{S_j} = -\log \|  f_{S_j}  \|^2_{g_{S_j}},
$$
and define 
$$r_D : = \prod_i r_{D_i} \,\,\,\,\,\, {\rm  and} \,\,\,\,\,\, r_S : =  \prod_j  r_{S_j}  .$$ 
The functions $r_D^\alpha$ and $\log r_D$ (resp. $r_S^\alpha$ and $\log r_S$)  are locally $L^1$ on $Y$ for all $\alpha\in \Z$. Indeed, if we write locally $f_{D_i}=z_i\cdot \tilde s_i$ (resp. $f_{S_j}=z_j\cdot \tilde s_j$), where $ \tilde s_i$ (resp. $\tilde s_j$) trivializes $\sO_Y(D_i)$ (resp. $\sO_Y(S_j)$) and $z_i$ (resp. $z_j$) is a coordinate, then locally 
$$r_{D_i}=-\log |z_i|^2-\log\|\tilde s_i\|^2_{g_{D_i}},$$
and similarly for $r_{S_j}$.  When $\alpha<0$, $r_D^\alpha$ and $r_S^\alpha$ are bounded and hence continuous on $Y$ 
(and smooth outside of $D$, resp. $S$).

We now fix an ample line bundle $\sL$ on $Y$ with a smooth hermitian metric 
$g$, so that its curvature $F(\sL, g)$ is positive. The metric $g$ induces a hermitian metric $g^{-1}$ on $\sL^{-1}$.
For  $\alpha \in \mathbb N$, we define
$$g_{\alpha} = g\cdot (r_{D}\cdot r_S)^{\alpha}$$ to be a singular metric on $\sL$. 
There is an induced singular metric $g_{\alpha}^{-1}=g^{-1}\cdot (r_{D}\cdot r_S)^{-\alpha}$ 
on $\sL^{-1}$. With this notation, we have 

\begin{align}\label{B}   
F(\sL, g_\alpha)  & =  F(\sL, g) -\alpha\cdot \sum_i r_{D_i}^{-1}\cdot
     F(\sO_Y(D_i), g_{D_i})  \\ \nonumber
                                 & \;\;  - \alpha \cdot \sum_j r_{S_j}^{-1} \cdot F(\sO_Y(S_j), g_{S_j}) \\ \nonumber 
                                 &  \; \;  +  \alpha\frac{\sqrt{-1}}{2\pi} \sum_i r_{D_i}^{-2}\cdot  \partial r_{D_i}\wedge \bar\partial r_{D_i}    
                                 \\ \nonumber 
                                 & \; \;  +   \alpha\frac{\sqrt{-1}}{2\pi} \sum_j r_{S_j}^{-2} \cdot \partial r_{S_j}\wedge \bar\partial r_{S_j}.   
\end{align}
Next we define a continuous $(1,1)$-form $\eta_\alpha$ on $Y$ by the formula 
$$
\eta_\alpha :=   F(\sL, g) 
-\alpha\cdot \sum_i r_{D_i}^{-1} \cdot F(\sO_Y(D_i), g_{D_i})  
                                - \alpha \cdot \sum_j r_{S_j}^{-1} \cdot F(\sO_Y(S_j), g_{S_j}),
$$
where we use the fact that $r_D^{-1}$ and $r_S^{-1}$ are continuous on $Y$.
As $Y$ is compact, after rescaling $f_{D_i}$ and $f_{S_j}$, we can arrange for the contributions of the last two terms 
in $\eta_\alpha$ to be sufficiently small
for $\eta_\alpha$ to be a continuous and positive definite $(1,1)$-form on $Y$.  On the other hand, one can easily check that 
$$\frac{\sqrt{-1}}{2\pi}\partial r_{D_i} \wedge \bar\partial r_{D_i} \,\,\,\,\,\, {\rm and}  \,\,\,\,\,\, \frac{\sqrt{-1}}{2\pi}\partial r_{S_j}\wedge \bar\partial r_{S_j}$$ 
are smooth and semi-positive $(1,1)$-forms on $Y\smallsetminus D_i$ (resp. $Y\smallsetminus S_j$) for all $i$ and $j$.

\begin{lemma}\label{lem:conical}
In the above setting, for each $\alpha \in \mathbb N$, after rescaling $f_{D_i}$ and $f_{S_i}$ there is a continuous,  positive definite, hermitian form $w_\alpha$  on $\sT_Y(-\log D)$ such that 
$$
F(\sL,  g_\alpha)|_{Y\smallsetminus (D+S)}  \geq  
r_D^{-2}\cdot w_\alpha|_{Y\smallsetminus (D+S)}.
$$
\end{lemma}

\begin{proof} 
After suitable rescaling, we may assume that  $r_{D_i}\geq 1$ for all $i$.
Since $\eta_{\alpha}$ is positive on $Y$, and $\frac{\sqrt{-1}}{2\pi}\partial r_{D_i}\wedge \bar\partial r_{D_i} $ (resp. $\frac{\sqrt{-1}}{2\pi}\partial r_{S_j}\wedge \bar\partial r_{S_j}$) are semi-positive on $Y\smallsetminus D_i$ (resp. $Y\smallsetminus S_j$), 
using (\ref{B}) we obtain
\begin{align*}
F(\sL, g_\alpha)|_{Y\smallsetminus (D+S)} & \geq \big( \eta_\alpha +  
\alpha\frac{\sqrt{-1}}{2\pi}\sum r_{D_i}^{-2} \cdot \partial r_{D_i}\wedge \bar\partial r_{D_i}\big)|_{Y\smallsetminus (D+S)}
                                    \\ 
                               & \geq r_D^{-2}\cdot \big( \underbrace{\eta_\alpha +\alpha\frac{\sqrt{-1}}{2\pi} \sum\partial r_{D_i} \wedge \bar\partial r_{D_i}}_{
                                    w_\alpha } \big)|_{Y\smallsetminus (D+S)}.
\end{align*}
Now, the claim that $w_\alpha$ is positive-definite on $\sT_Y(-\log D)$ follows from 
the fact that $\eta_\alpha$ is positive on $\sT_Y$
and that $\sum\partial r_{D_i} \wedge \bar\partial r_{D_i}$ is positive definite along the vector 
fields that are tangent to $D$, cf.~\cite[Claim.~7.2]{Vie-Zuo03a}.                                   
\end{proof}

\medskip

We now switch focus to Hodge metrics. Recall that we are always dealing with polarizable VHS with quasi-unipotent monodromy along simple normal crossings boundary.
The following lemma translates the results of \cite{CKS} on the singularities of Hodge metrics in the unipotent case  
to this setting, and will be important for the proof of Proposition \ref{prop:BHHiggs}.

\begin{lemma}[Estimates for Hodge metrics; the quasi-unipotent case]
\label{lem:QUnipotent}
Suppose $\Delta^n$ is a polydisk with coordinates $(z_1,\ldots, z_n)$. Let $\V$ be a polarized VHS on the open set 
$U=\Delta^n\smallsetminus(\{ (z_1, \ldots, z_k) |  \prod_{i=1}^k z_i =0 \})$, $k \le n$, 
with quasi-unipotent monodromies along each 
$\{z_i=0\}$, 
and denote by 
$\sE_\bullet$ the Higgs bundle associated to the Deligne extension of $\V$ with eigenvalues in $[0, 1)$.
Then the Hodge metric induced by the polarization has at most logarithmic singularities along each $z_i$, for $i=1,\dots,k$; that is, there exists an integer $d > 0$ such that 
for any section $e$ of $\sE_\bullet$ locally we have 
\begin{equation}\label{eq:QUnipotent}
\| e \|^2_{h} \leq  C\cdot  \prod_{i=1}^k \big(  -\log |z_i| \big)^d
\end{equation}
for some constant $C = C(e)\in \mathbb R_{>0}$.
\end{lemma}

\begin{proof}
Let $L$ be the local system underlying $\V$, with monodromy $\Gamma_i$ along $z_i$, for $i=1,\dots, k$. Since the $\Gamma_i$ commute pairwise, we have 
\[L=\bigoplus_\alpha L_\alpha.\]
as the simultaneous (generalized) eigenspace decomposition with respect to the monodromy actions. 
Thus the monodromy action $\Gamma_i$ on $L_{\alpha=(\alpha_1,\ldots,\alpha_k)}$ has a unique eigenvalue 
$e^{-2\pi\sqrt{-1}\alpha_i}$. By the quasi-unipotent assumption, we can assume all $\alpha_i$ are rational numbers contained in $[0,1)$. By the lower semicontinuity of rank functions of matrices, the above decomposition induces a decomposition of polarized variations of Hodge structure
\[\V=\bigoplus \V_\alpha,\]
and hence a decomposition of Higgs bundles 
\[\sE_\bullet=\bigoplus \sE_{\bullet}^\alpha,\]
where $\sE_\bullet^\alpha$ is the Higgs bundle associated to the Deligne extension of $\V_\alpha$ with eigenvalues in $[0, 1)$. (Note that the extension of $\V_\alpha$ has only one eigenvalue along each $z_i$.)

If $\alpha=(\alpha_1,\ldots,\alpha_k)\neq 0$, then we can write $\alpha_i=\frac{p_i}{q_i}$ for some non-negative integers $p_i<q_i$. Now, let $g: \Delta^n\rarr \Delta^n$ be the branched covering given by
\begin{equation} \label{bcover}
g^*z_i = \left\{
             \begin{array}{lcl}
             {w_i^{q_i}} &\text{if} &i=1,\dots, k \\
             {w_i} & &\text{otherwise}, 
             \end{array}  
        \right.
\end {equation}
where $(w_1,\dots,w_n)$ define a coordinate system on the domain of $g$. It follows that the monodromies of $g^*\V_\alpha$ along $w_i$ are unipotent. By comparing the eigenvalues of the residues upstairs, we have 
\[g^*\sE_{\bullet}^\alpha=\prod_{i=1}^kw_i^{p_i}\cdot\sE^{\alpha}_{g, \bullet},\]
where $\sE^{\alpha}_{g, \bullet}$ is the Higgs bundle associated to the Deligne canonical extension 
of $g^* \V_\alpha$. 
Since the Hodge metric on $\sE^{\alpha}_{g, \bullet}$ has logarithmic singularities (see \cite[\S 5.21]{CKS}), 
for a section $e$ of $\sE_{\bullet}^\alpha$ we know that
\[\|  e \|^2_{h} \leq  C\cdot  \prod_{i=1}^k \big(| z_i|^{\alpha_i} \cdot (-\log |z_i|)^{d_i} \big)\leq C\cdot  \prod_{\alpha_i=0} (-\log |z_i|)^{d_i},\]
for some positive integers $d_i>0$. The Hodge metric has logarithmic singularities along the $z_i$ whenever $\alpha_i=0$.
In particular, we get the inequality~($\ref{eq:QUnipotent}$) when $\alpha = (\alpha_1, \ldots, \alpha_k) \neq 0$.

On the other hand, if $\alpha=(\alpha_1,\ldots,\alpha_k)=0$, then we know that the monodromies of $\V_\alpha$ are unipotent. 
Therefore, again thanks to \cite{CKS}, the Hodge metric on $\sE_{\bullet}^\alpha$ has logarithmic singularities along each 
$z_i$, as required. 
\end{proof}

\begin{remark}
The above lemma implies that the Hodge metric is a singular metric on the vector bundle $\sE_\bullet$.
\end{remark}

Let us now return to the setting described at the beginning of this section, and suppose in addition that $\sE_\bullet$ is the Higgs bundle associated to the Deligne extension of a VHS on $Y\smallsetminus (D+S)$, with eigenvalues in $[0, 1)$.
We define a singular metric $h^{\alpha}_g$ on the vector bundle $\sL^{-1}\otimes \sE_\bullet$ by 
\[h^{\alpha}_g=g^{-1}_\alpha \otimes h,\]
where $h$ is the Hodge metric on $\sE_\bullet$.

\begin{corollary}\label{cor:twistHm}
For all $\alpha\gg 0$, the singular metric $h^{\alpha}_g$ is locally bounded. 
\end{corollary}
\begin{proof}
Assume that in local coordinates $D+S$ is given by $z_1\cdots z_{k+\ell} =0$. By construction, the singular metric $g^{-1}_\alpha$ degenerates to $0$ at a rate proportional to 
$$(r_D\cdot r_S)^{-\alpha}=\prod_{i=1}^k(-\log |z_i|^2-\log\|\tilde s_i\|^2_{g_{D_i}})^{-\alpha}\cdot\prod_{i=k+1}^{k + \ell}(-\log |z_i|^2-\log\|\tilde s_i\|^2_{g_{S_i}})^{-\alpha}.$$ 
On the other hand, the Hodge metric $h$ on $\sE_\bullet$ blows up to infinity along $z_i=0$ bounded by a quantity proportional to 
$\prod (-\log |z_i|^2)^{d}$, for some fixed $d > 0$, thanks to Lemma~\ref{lem:QUnipotent}. Hence, the metric $h_g^{\alpha}$ is bounded by a quantity proportional to 
$$(r_D\cdot r_S)^{-\alpha+d}\cdot\prod_{i=1}^{k}\big(\frac{-\log |z_i|^2}{-\log |z_i|^2-\log\|\tilde s_i\|^2_{g_{D_i}}}\big)^d\cdot\prod_{j=k+1}^{k + \ell}\big(\frac{-\log |z_j|^2}{-\log |z_j|^2-\log\|\tilde s_j\|^2_{g_{S_j}}}\big)^d.$$ 
When $\alpha> d$ the above product is bounded. The compactness of $Y$ gives the conclusion.
\end{proof}

\subsection{An application of the singular Ahlfors-Schwarz Lemma}
\label{subsect:AS}
In this section we establish the key technical ingredient. This is done by applying the tools discussed in the previous section to the base spaces of families
of varieties, via the Hodge-theoretic set-up provided by the constructions 
in \S\ref{sect:HodgeModules}, especially those in  Proposition~\ref{prop:summary}.

\begin{proposition}\label{prop:BHHiggs}
In the situation of Proposition~\ref{prop:summary}, 
the morphism 
$$
\tau_{(\gamma, 1)}: \sT_{\mathbb C} \longrightarrow \gamma^*(\sL^{-1} 
  \otimes \sE_{1})
$$
induced by the entire curve $\gamma: \mathbb C \to  Y\smallsetminus D$ is identically zero.
\end{proposition}

\begin{proof}
The proof will be by contradiction. First we note that, assuming that $\tau_{(\gamma, 1)}$ is non-trivial, the following claim holds.

\begin{claim}\label{claim:inject}
There exist: 
\begin{enumerate}
\item  integers $m > 0$ and $p > 0$,
\item an ample line bundle $\sH$ on $Y$, and
\item a Higgs bundle $(\sE^\prime_\bullet, \theta^\prime_\bullet)$ on $Y$ underlying the Deligne extension with eigenvalues 
in $[0, 1)$ of a VHS defined outside of $D+S$
\end{enumerate}
such that there is a non-trivial (hence injective) morphism $\tau_m \colon \sT_{\mathbb C}^{\otimes m} \to \gamma^*(\sH^{-1} \otimes \sE^\prime_p)$ factoring as
\begin{equation}\label{eq:Injection}
\tau_m \colon \sT_{\mathbb C}^{\otimes m} \xrightarrow{d\gamma^{\otimes m}} \gamma^*\big(\bigotimes^{m}  \sT_Y (- {\rm log}D) \big)
\longrightarrow \gamma^*\sH^{-1} \otimes \sN^\prime_{(\gamma,p)} \hookrightarrow  \gamma^*(\sH^{-1} \otimes \sE^\prime_p),
\end{equation}
where $\sN^\prime_{(\gamma,\bullet)} = \ker \theta^\prime_{(\gamma,\bullet)}$, with $\theta^\prime_{(\gamma,\bullet)}$  the Higgs field of $\gamma^*\sE^\prime_\bullet$ (see Definition~\ref{pullbackhb}).
\end{claim}

\noindent
\emph{Proof of Claim~\ref{claim:inject}.}
By construction, for all sufficiently large $k$ we have $\tau_{(\gamma, k)} =0$.
We set
$$
p :=\max\{k\text{ $|$ } \tau_{(\gamma, k)}\neq 0\}.
$$ 
By assumption (the injectivity of $\tau_{(\gamma, 1)}$), 
we have $p\ge 1$. 
On the other hand, we know that $\tau_{(\gamma, p+1)}$ factors as
\[\tau_{(\gamma, p+1)}: \sT_{\mathbb C}^{\otimes (p+1)}\xrightarrow{\text{Id}\otimes\tau_{(\gamma, p)} } \sT_{\mathbb C}\otimes  \gamma^*\sL^{-1} \otimes  \gamma^*\sE_{p}\to  \gamma^*\sL^{-1} \otimes  \gamma^*\sE_{p+1}(P),\]
where the last map is  induced by the $\sA_\CC(-\log P)$-module structure on $  \gamma^*\sE_{\bullet}$, with $P=\gamma^{-1}(S)$.
(Note that in fact its image lands in $\gamma^*\sL^{-1} \otimes \gamma^*\sE_{p+1}$ as required, due to the fact that in the definition, see Proposition \ref{prop:summary}, we factor through the Higgs field of $\sF_\bullet$, 
which does not have poles along $S$.) Since $\tau_{(\gamma, p+1)}=0$, we obtain that $\tau_{(\gamma, p)}$ injects 
$\sT_{\mathbb C}^{\otimes p}$ into $\gamma^* \sL^{-1} \otimes 
\sN_{(\gamma, p)}$, where $\sN_{(\gamma, \bullet)} = \ker \theta_{(\gamma, \bullet)}$, with $\theta_{(\gamma, \bullet)}$ the induced Higgs field of $\gamma^*\sE_\bullet$.  
Thus we have a nontrivial composition of morphisms
$$
\tau_{(\gamma, p)} \colon \sT_{\mathbb C}^{\otimes p} \xrightarrow{d\gamma^{\otimes p}} \gamma^*\big(\bigotimes^{p}  \sT_Y (- {\rm log}D) \big)
\longrightarrow \gamma^*\sL^{-1} \otimes \sN_{(\gamma, p)} \hookrightarrow  \gamma^*(\sL^{-1} \otimes \sE_{p}).
$$

Now since $\sL$ is big and nef, there exists $q > 0$ and an ample line bundle $\sH$ such that $\sH \subseteq \sL^{\otimes q}$. 
Similarly to the proof of \cite[Lemma~6.5]{Vie-Zuo03a}, we consider the Higgs bundle $(\sE^\prime_\bullet, \theta^\prime_\bullet)$ on $Y$ given by

\begin{equation}\label{new_field}
\sE^\prime_\bullet = \sE_{\bullet}^{\otimes q} \,\,\,\,\,\,{\rm and} \,\,\,\,\,\, \theta^\prime_\bullet \colon \sE_{ \bullet}^{\otimes q} \to \sE_{ \bullet +1}^{\otimes q} \otimes \Omega_{Y}^1 (D+S),
\end{equation}
$$ 
\theta^\prime_\bullet = 
 \theta_{\bullet} \otimes {\rm id}_{\sE} \otimes \cdots \otimes  {\rm id}_{\sE} + 
 {\rm id}_{\sE} \otimes \theta_{\bullet} \otimes \cdots \otimes  {\rm id}_{\sE} +\cdots + {\rm id}_{\sE} \otimes \cdots \otimes  {\rm id}_{\sE} \otimes \theta_{\bullet}.
$$

\noindent
As noted in \emph{loc. cit.}, this Higgs bundle corresponds to the locally 
free extension $V^\prime$ to $Y$ of the bundle coming from the VHS $\V^{\otimes q}$ on $Y \smallsetminus (D+S)$, where $\V$ is the VHS 
underlying $\sE_{ \bullet}$. The induced connection on $V^\prime$ has residues with eigenvalues in $\QQ_{\ge 0}$, 
and therefore $V^\prime$ is contained in $V^{''}$, the Deligne extension with eigenvalues in $[0,1)$ (see \cite[Prop. 4.4]{PW16}). Therefore, without loss 
of generality, in the paragraph below we can assume that $(\sE^\prime_\bullet, \theta^\prime_\bullet)$ is in fact  the Higgs bundle 
associated to this extension. Note moreover that when pulling back by $\gamma$, the above construction implies that we have an inclusion of logarithmic Higgs bundles on $\CC$ 
\begin{equation}\label{newinclusion1}
\big((\gamma^*\sE_\bullet)^\prime, \theta^\prime_\bullet \big) \subseteq \big(\gamma^*\sE^\prime_\bullet, \theta^\prime_{(\gamma, \bullet)}\big),
\end{equation}
where the Higgs bundle on the left is the analogue for $\gamma^*\sE_\bullet$ of the construction in (\ref{new_field}).

Finally, let $m : = pq$. Raising $\tau_{(\gamma, p)}$, seen as the composition of morphisms above, to the $q$-th tensor power, 
gives rise to a new nontrivial composition of morphisms:
$$\tau_m \colon \sT_{\mathbb C}^{\otimes m} \xrightarrow{d\gamma^{\otimes m}} \gamma^*\big(\bigotimes^{m}  \sT_Y (- {\rm log}D) \big)
\longrightarrow  \gamma^*\sH^{-1} \otimes \sN_{(\gamma, p)}^{\otimes q} \hookrightarrow  \gamma^*\sH^{-1} \otimes \gamma^*\sE^\prime_p,$$
where we used the inclusion of $\sL^{\otimes - q}$ into $\sH^{-1}$. In addition, the formula for the Higgs field on the left hand side
of ($\ref{newinclusion1}$) (cf. (\ref{new_field})) implies immediately that $ \sN_{(\gamma, p)}^{\otimes q} \subseteq \sN^\prime_{(\gamma,p)}$, where we recall that $ \sN^\prime_{(\gamma,p)}=\ker \theta^\prime_{(\gamma, p)}$, so $\tau_m$ does factor as in (\ref{eq:Injection}).
This concludes the proof of Claim~\ref{claim:inject}.\qed

\medskip
We continue with the proof of Proposition \ref{prop:BHHiggs}. 
For simplicity, after renaming again $\sH$ and $\sE^\prime_\bullet$ in Claim~\ref{claim:inject} by 
$\sL$ and $\sE_\bullet$, we assume from now on that $\sL$ is ample, and that the morphism $\tau_m$ is given as
\begin{equation}\label{eq:Injection1}
\tau_m \colon \sT_{\mathbb C}^{\otimes m} \xrightarrow{d\gamma^{\otimes m}} \gamma^*\big(\bigotimes^{m}  \sT_Y (- {\rm log}D) \big)
\longrightarrow \gamma^*\sL^{-1} \otimes \sN_{(\gamma,p)} \hookrightarrow  \gamma^*\sL^{-1} \otimes \gamma^*\sE_p.
\end{equation}

Our aim is now to extract a contradiction from the existence of a non-trivial such morphism, by showing that $\mathbb C$ 
inherits a singular metric $h_{\CC}$ satisfying the distance decreasing
property for any holomorphic map $g\colon (\mathbb D, \rho) \to (\CC, h_{\CC})$, 
that is $d_{h_{\mathbb C}}(g(x), g(y))  \leq A\cdot  d_{\rho}(x,y)$, 
where $\rho$ is the Poincar\'e metric on the unit disk, and $A\in \mathbb R_{>0}$.
Since the Kobayashi pseudo-metric is larger than any such distance function, 
this forces it to be non-degenerate, contradicting the fact that on $\CC$ it is identically zero. 
For background on  this material, see for instance \cite[Chapt.~IV, Sect.1]{Kob05}.

Note first that, according to the Ahlfors-Schwarz lemma for (locally integrable) 
singular metrics over curves, 
cf.~\cite[Lem.~3.2]{Dem97}, any singular metric verifying, for some $B\in \mathbb R_{> 0}$, the inequality 
\begin{equation}\label{eq:Schwarz}
F(\sT_\CC, h_{\CC})  \leq - B \cdot w_{h_{\CC}}
\end{equation}
in the sense of currents, 
satisfies the above distance decreasing property. 
(Here $\omega_{h_{\CC}} = \frac{\sqrt{-1}}{2\pi} \|\partial_t\|^2_{h_{\mathbb C}}d t\wedge d\bar t$ denotes 
the fundamental form of the metric $h_{\CC}$, which we have assumed to be a $(1,1)$-current, where $t$ is the coordinate of $\CC$.)
Therefore, to conclude, it suffices to construct a metric $h_{\CC}$ on $\CC$ verifying the inequality~($\ref{eq:Schwarz}$). 
We next proceed to construct such a metric.

\medskip

We first fix a smooth metric $g$ on $\sL$, so that the curvature form $F(\sL, g)$ is positive. 
Following the notation in \S\ref{subsect:Sing}, this induces a 
singular metric $g_\alpha$ on $\sL$, and a singular metric $h^{\alpha}_g = g^{-1}_\alpha \otimes h$ on $\sL^{-1}\otimes \sE_\bullet$, 
where we fix an  $\alpha\gg 0$ as in Corollary~\ref{cor:twistHm}. Consequently $\gamma^*h^\alpha_g$ is a singular metric on 
$\gamma^*(\sL^{-1}\otimes \sE_\bullet)$, and the $m$-th root of its pullback, 
$$h_\CC:= (\tau_m^*\gamma^*h^\alpha_g )^{\frac{1}{m}},$$ 
defines a singular metric on (the trivial line bundle) $\sT_\CC$.

Similarly, we have the continuous positive definite hermitian form $\omega_\alpha$ on $\sT_Y(-\log D)$ as in Lemma~\ref{lem:conical}, and so $\gamma^*\omega_{\alpha}$ induces a singular metric on 
$\gamma^*\sT_Y(-\log D)$), and hence also a singular metric on $\sT_\CC$ through the differential map.  For the next claim, recall that $P = \gamma^{-1}(S) \subset \CC$.

\begin{claim}\label{claim:1}
We have $m\cdot F(\sT_\CC,h_\CC)|_{\CC\smallsetminus P}\le -\gamma^*(r^{-2}_D)\cdot\gamma^*\omega_{\alpha}|_{\CC\smallsetminus P}$, in the sense of currents.
\end{claim}

\noindent
\emph{Proof of Claim~\ref{claim:1}.}
Note that $F(\sT_\CC,h_\CC)$ makes sense as a current on $\CC \smallsetminus P$. The proof of the claim will also imply that 
it is indeed a current everywhere on $\CC$, as we explain afterwards.

Denote by $\sB$ the saturation of $\tau_m(\sT^{\otimes m}_\CC)$ inside $\gamma^*(\sL^{-1}\otimes \sE_\bullet)$, so that 
\[\sB \simeq \sT^{\otimes m}_\CC (G), \]
where $G\ge 0$ is a divisor on $\CC$. 
Since $\tau_m$ factors through $\gamma^*\sL^{-1} \otimes \sN_{(\gamma,p)}$, we know that $\theta_{(\gamma,\bullet)}(\gamma^*\sL\otimes \sB)=0$. Recall that as a consequence of Griffiths' curvature estimates for Hodge metrics,  
it is well known (see e.g. \cite[Lem.~1.1]{Vie-Zuo01} and the references therein) that the Hodge metric restricted to 
any subbundle inside the kernel of the Higgs field associated to a VHS has semi-negative curvature. We thus conclude that
\[F(\sB,\gamma^*h^{\alpha}_g|_\sB)|_{\CC\smallsetminus P}+\gamma^*F(\sL, g_\alpha)|_{\CC\smallsetminus P} \le 0,\]
and since 
$$(\sT^{\otimes m}_\CC\otimes \gamma^*\sL) ( G) \simeq\sB \otimes\gamma^*\sL,$$
this implies
\begin{equation}\label{eq:inequa1}
m\cdot F(\sT_\CC,h_\CC)|_{\CC\smallsetminus P}+\gamma^*F(\sL, g_\alpha)|_{\CC\smallsetminus P} \le 
\end{equation}
$$\le F(\sB,\gamma^*h^{\alpha}_g|_\sB)|_{\CC\smallsetminus P}+\gamma^*F(\sL, g_\alpha)|_{\CC\smallsetminus P}\le0.
$$
Now the statement follows from Lemma~\ref{lem:conical}.
\qed

\medskip

As mentioned above, the proof of the claim also implies that $F(\sT_\CC,h_\CC)$ is a current on $\CC$. Indeed, from construction, we know $F(\sL, g_\alpha)$ is a $(1,1)$-current and $F(\sL, g_\alpha)|_{Y\smallsetminus (D+S)}$ is positive.  Hence, by \eqref{eq:inequa1}, we know $F(\sT_\CC,h_\CC)|_{\CC\smallsetminus P}$ is negative; or equivalently, $\log \|\partial_t\|^2_{h_\CC}$ is subharmonic on $\CC\smallsetminus P$. Since $h_\CC$ is locally bounded (see Corollary~\ref{cor:twistHm}), $\log \|\partial_t\|^2_{h_\CC}$ extends to a subharmonic function on $\CC$ (see \cite[Thm. 5.23]{DemaillyBook}), and so  $F(\sT_\CC,h_\CC)$ is
a negative current.

\medskip

Next we fix a polydisk neighborhood $\Delta^n\subseteq Y$. The continuous metric $\|\cdot \|_{\omega_{\alpha}}$ on 
$\sT_Y(-\log D)$ given by $\omega_{\alpha}$ induces a metric on $\bigotimes^{m} \sT_Y(-\log D)$. We also fix 
an orthonormal basis $\{\psi_1,\dots,\psi_N\}$ of continuous sections of $\bigotimes^{m} \sT_Y(-\log D)|_{\Delta^n}$ with respect to the induced metric. (By abuse of notation, we use $\bigotimes^{m} \sT_Y(-\log D)|_{\Delta^n}$ even when considering the associated sheaf of
continuous sections.)

We fix a holomorphic basis $\{e_1, e_2,\dots,e_M\}$ of $\sL^{-1}\otimes\sE_p|_{\Delta^n}$ as well. We write
\begin{equation}\label{eq:tr1}
\tilde\tau_m(\psi_i)=\sum_j b_i^j\cdot e_j
\end{equation}
for some continuous functions $b_i^j$ on $\Delta^n$, where 
\[\tilde\tau_m \colon \bigotimes^{m} \sT_Y(-\log D)\to \sL^{-1}\otimes \sE_p,\]
and we also write
\begin{equation}\label{eq:tr4}
d\gamma^{\otimes m}(\partial^m_t|_{\gamma^{-1}(\Delta^n)})=\sum_i c_i\cdot \gamma^* \psi_i,
\end{equation}
for some continuous (complex valued) functions $c_i$.

\begin{claim}\label{claim:2}
We have $\gamma^*(r^{-2}_D)\cdot\gamma^*\omega_{\alpha}\ge B\cdot\omega_{h_\CC}$ in the sense of currents on $\CC$, 
for some $B>0$.
\end{claim}

\noindent
\emph{Proof of Claim~\ref{claim:2}.}
Since $\sT_\CC$ is trivialized by $\partial_t$ globally, it it enough to show 
\[\gamma^*(r^{-2m}_D)\cdot\|d\gamma^{\otimes m}(\partial^m_t)\|_{\gamma^*\omega_{\alpha}}\ge B\cdot\|\tau_m(\partial^m_t)\|_{\gamma^*h^\alpha_g}.\]
By the compactness of $Y$, it is enough to prove the inequality locally on neighborhoods of the form $\gamma^{-1}(\Delta^n)$, 
with $\Delta^n \subset Y$ as above.   

First, since $\{\psi_1,\dots,\psi_N\}$ is an orthonormal basis, by \eqref{eq:tr4} we see that
\begin{equation}\label{eq:ortho}
\|d\gamma^{\otimes m}(\partial^m_t|_{\gamma^{-1}(\Delta^n}))\|_{\gamma^*\omega_{\alpha}}=
\big(\sum_i |c_i|^2\big)^{\frac{1}{2}}.
\end{equation}
By \eqref{eq:tr1} and \eqref{eq:tr4}, we also have
\[\|\tau_m(\partial^m_t|_{\gamma^{-1}(\Delta^n}))\|_{\gamma^*h^{\alpha}_g}=\|\sum_ic_i\sum_j\gamma^*\big(b_i^j\cdot e_j\big)\|_{\gamma^*h^\alpha_g}.\]
On the other hand, by the Cauchy-Schwarz inequality, we have
\[\|\sum_ic_i\sum_j\gamma^*\big(b_i^j\cdot  e_j)\big)\|_{\gamma^*h^{\alpha}_g}\le \big(\sum_i |c_i|^2\big)^{\frac{1}{2}}\cdot \big(\sum_i\gamma^*\|\sum_j\big(b_i^j\cdot  e_j\big)\|^2_{h^{\alpha}_g}\big)^{\frac{1}{2}}.\]
By Corollary~\ref{cor:twistHm}, we know that $h^{\alpha}_g$ is bounded over $\Delta^n$ by a quantity proportional to 
$$(r_D\cdot r_S)^{-\alpha+d}\cdot\prod_{i=1}^{k}\big(\frac{-\log |z_i|^2}{-\log |z_i|^2-\log\|\tilde s_i\|^2_{g_{D_i}}}\big)^d\cdot\prod_{j=k+1}^{k + \ell}\big(\frac{-\log |z_j|^2}{-\log |z_j|^2-\log\|\tilde s_j\|^2_{g_{S_j}}}\big)^d,$$
for some fixed $d > 0$. Therefore, we have
\[\|\tau_m(\partial^m_t|_{\gamma^{-1}\Delta^n})\|_{\gamma^*h^\alpha_g}\le \frac{1}{B}\cdot \gamma^*(r_D^{\frac{d- \alpha}{2}})\cdot \big(\sum |c_i|^2\big)^{\frac{1}{2}}
\]
for some $B>0$ and $\alpha$ sufficiently large (recall that $r_S^\gamma$ is bounded for $\gamma < 0$).
This implies the conclusion,  given (\ref{eq:ortho}) and the fact that earlier we have chosen our scaling so that $r_D \ge 1$.
\qed

\medskip

Finally, the inequality~\eqref{eq:Schwarz} follows from Claim~\ref{claim:1}, Claim~\ref{claim:2}, and the fact that if 
the inequality
$$
F(\sT_\CC, h_{\mathbb C})|_{(\mathbb C\smallsetminus P)} \leq - B \cdot  \big(\omega_{h_{\mathbb C}}|_{(\mathbb C\smallsetminus P)}\big)
$$
holds as currents for some $B>0$, then we also have 
$$
F(\sT_\CC, h_{\mathbb C}) \leq  - B \cdot  \omega_{h_{\mathbb C}},
$$
as currents on $\mathbb C$. But this is an easy consequence of the negativity of $F(\sT_\CC, h_{\CC})$, together with the 
continuity of $\omega_{h_{\mathbb C}}$.

\end{proof}

\subsection{Some further background}
In this section we collect a few useful facts regarding entire maps on the one hand, and families with maximal variation on the other.

\subsubsection{Algebraic degeneracy to Brody hyperbolicity}
In \S\ref{subsect:summary} we observed that the Hodge theoretic constructions of 
\S\ref{subsect:Hodge} are valid as long as we replace the initial family $f_U\colon U\to V$
by a birational model, compactified by the family $f \colon X\to Y$ in Proposition~\ref{prop:summary}.  We recall below, following 
\cite[\S1]{Vie-Zuo03a}, that the study of the hyperbolicity properties can be reduced to investigating algebraic nondegeneracy
on such models.

\begin{lemma}[{\cite[Lem.~1.2]{Vie-Zuo03a}}]
\label{lem:BlowUp}
Let $\gamma\colon \mathbb C \to V$ be an entire curve with a Zariski-dense image,
and $\mu\colon \wtilde V \to V$ a birational morphism. 
Then the map $(\mu^{-1}\circ \gamma)$ extends to a holomorphic map $\wtilde \gamma: \mathbb C 
\to \wtilde V$.
\end{lemma}

\begin{proposition}[Reduction of Brody hyperbolicity to algebraic degeneracy]
\label{prop:reduction}
Let $P_h$ be a coarse moduli space of polarized manifolds, as in the Introduction,
 and $V$ and $Y$ as in Proposition~\ref{prop:summary}.
\begin{enumerate}
\item \label{item:reduce1} The image of $\gamma\colon \mathbb C\to V$ is algebraically 
degenerate if and only if  the induced morphism $\wtilde \gamma \colon \mathbb C\to \widetilde V$ defined in Lemma~\ref{lem:BlowUp}
is so.

\item \label{item:reduce2} To prove the Brody hyperbolicity of $P_h$, in the sense of Theorem \ref{thm:BHMain}, 
it suffices to show that for 
every smooth quasi-projective variety $V$ with a generically finite morphism $V\to P_h$, 
every entire curve $\mathbb C\to V$ is algebraically degenerate. 
\end{enumerate}
\end{proposition}

\begin{proof}
Item~\ref{item:reduce1} is the direct consequence of Lemma~\ref{lem:BlowUp}.
For Item~\ref{item:reduce2}, note that given a quasi-finite morphism $W\to P_h$
from a variety $W$, and $\gamma \colon \mathbb C\to W$, the restriction $W'$ of $\Im(\gamma)$
 to the Zariski closure $W'$ of $\Im(\gamma)$ is also quasi-finite. 
Furthermore, we can desingularize $W'$ by $\mu \colon \wtilde W' \to W'$, and by  
\ref{item:reduce1}, the degeneracy of the induced map $\mathbb C\to \wtilde W'$
 is equivalent to the fact that $\gamma$ is constant.
 \end{proof}

Therefore, to prove Theorem~\ref{thm:BHMain} on the Brody hyperbolicity of
$P_h$, it suffices to establish Theorem~\ref{thm:MaxVar}.

\subsubsection{More on families with maximal variation}
\label{scn:max_var}
We recall a few facts about families with maximal variation that were established by Koll\'ar \cite{Kollar87}.
Here $f\colon U \to V$ is a smooth projective morphism of smooth varieties, with fibers of non-negative Kodaira dimension. 

\begin{lemma}[{\cite[Cor.~2.9]{Kollar87}}]\label{countable}
If ${\rm Var} (f) = \dim V$, and if $v$ is a very general point of $V$, then for any analytic arc $\gamma \colon \Delta \to V$ passing through $v$, 
not all fibers of $f$ over $\gamma (\Delta)$ are birational. 
\end{lemma}

We will denote by $W \subset V$ the locus of points $v$ satisfying the property in Lemma \ref{countable}. In general we see that 
$W$ is the complement of a countable union of closed subsets of $V$. 
The following result says that when the fibers of $f$ are of general type, it is guaranteed to contain a Zariski open set
$V_0$.

\begin{lemma}[{\cite[Thm.~2.5]{Kollar87}}]
\label{lem:FiniteMap}
If the fibers of $f$ are of general type, then there exists an open subset $V_0 \subseteq V$ and a morphism $g\colon V_0 \to Z$ onto an algebraic
variety, such that for $v_1, v_2 \in V_0$ the fibers $U_{v_1}$ and  $U_{v_2}$ are birational if and only if $g(v_1) = g(v_2)$.
\end{lemma}

Indeed, when ${\rm Var} (f) = \dim V$,  in the lemma above we have $\dim Z = \dim V$, 
and the map $g$ is generically finite. 
Thus there exists a, perhaps smaller, dense open subset $\wtilde V_0 \subseteq V$, such that $\wtilde V_0 \subseteq W$ (namely the complement of the positive dimensional fibers of $g$).

\subsection{Algebraic degeneracy for base spaces of families of minimal varieties of general type}\label{subsect:Brody}
We are now ready to prove Theorem~\ref{thm:MaxVar}. 

\medskip

\noindent 
\emph{Proof of Theorem~\ref{thm:MaxVar}.}
We first show that every holomorphic curve $\gamma \colon \CC \to V$ is algebraically degenerate.
According to Proposition~\ref{prop:reduction}, Item~\ref{item:reduce1}, we can assume that $V = Y\smallsetminus D$ as in Proposition~\ref{prop:summary}. 

Recall that the mapping appearing in Proposition \ref{prop:BHHiggs} can be written as the composition 
$$ 
\tau_{(\gamma, 1)} \colon  \sT_{\mathbb C} \to \gamma^* \sT_Y(-\log D)  \to 
  \gamma^*\big({\sF_0}^{-1} \otimes \sF_1 \big) \hookrightarrow \gamma^* (\sL^{-1}\otimes 
       \sE_1).
$$
Now by Corollary~\ref{cor:GTExtra} we have a generic identification of 
$$\tau_1\colon \sT_Y(-\log D) \longrightarrow {\sF_0}^{-1} \otimes \sF_1$$ 
with the Kodaira-Spencer map of the family $f \colon X \to Y$, and so by base change the composition of the first two 
maps in the definition of  $\tau_{(\gamma, 1)}$ can be identified with the Kodaira-Spencer map of the induced family over 
$\CC$. If $\gamma (\CC)$ were dense, we would obtain a family with maximal variation over $\CC$, implying that this 
Kodaira-Spencer map is injective; indeed, over a curve it can only be injective or $0$, the latter case of course 
implying that the family is locally trivial. But this in turn implies that $\tau_{(\gamma, 1)}$ is injective, 
which contradicts Proposition~\ref{prop:BHHiggs}.

We now show the stronger statement that ${\rm Exc}(V)$ is a proper subset,  knowing that the algebraic degeneracy statement 
we just proved holds for any base of a family as in the theorem. 
Let $V_0$ be the Zariski open subset in Lemma~\ref{lem:FiniteMap}
and $\wtilde V_0$ be the subset of $V_0$ over which the morphism $g$ is finite. 
We claim that 
$$
 \Exc(V)  \subseteq V\smallsetminus \wtilde V_0.
$$
To see this, assume that there exists an entire curve $\gamma: \mathbb C \to V$ with 
$\gamma(\mathbb C)\cap \wtilde V_0 \neq \emptyset$, and denote by $W$ the Zariski closure of $\gamma (\CC)$ in $V$.
If $\gamma$ is not constant, then by definition the restriction of the family $f$ over $W$ has maximal variation. Furthermore, 
using again Proposition \ref{prop:reduction}, we can assume that $W$ is smooth. We then obtain a contradiction with the 
algebraic degeneracy of all maps $\CC \to W$.

\subsection{Algebraic degeneracy for surfaces mapping to moduli stacks of polarized varieties}
\label{subsect:GGLocus}
We now prove the stronger statements in the case when the base of the family is a smooth 
surface. We start with two basic lemmas about pulling back sheaf morphisms via $\gamma$, the first of 
which is immediate.

\begin{lemma}\label{pullback_free}
Let $\gamma \colon \CC \to V$ be a holomorphic map with Zariski dense image, where $V$ is an algebraic variety. If $\varphi\colon \sE \to \sF$ is an 
injective morphism of locally free $\shO_V$-modules, then $\gamma^*\varphi \colon \gamma^* \sE \to \gamma^* \sF$ is also injective.
\end{lemma}

\begin{lemma}\label{pullback_inclusion}
Let $\gamma \colon \CC \to V$ be a holomorphic map with Zariski dense image, where $V$ is a smooth algebraic surface. Let $Z$ be a $0$-dimensional 
local complete intersection subscheme of $V$. Then we have an inclusion $\gamma^* I_Z \hookrightarrow \sO_{\CC}$.
\end{lemma}
\begin{proof}
We can cover $\CC$ with the preimages of open subsets in $V$ on which $Z$ is given as $f_1 = f_2= 0$, where $f_1$ and $f_2$ are two non-proportional functions. Denoting by $D_1$ and $D_2$ the divisors of these two functions, so that $Z$ is the scheme theoretic intersection $D_1 \cap D_2$, we can thus assume that we have a Koszul complex 
$$0 \to \shO_V ( - D_1 - D_2) \to \shO_V ( - D_1) \oplus \shO_V ( - D_2) \to I_Z\to 0.$$
Pulling back this sequence by $\gamma$, we still have a short exact sequence, as the first map degenerates only at the points of $Z$.
Therefore we have a commutative diagram 
\[\begin{tikzcd}
& 0 \dar & 0 \dar & & \\
0 \rar & \gamma^*\shO_V ( - D_1 - D_2)  \rar \dar  & \gamma^* \shO_V ( - D_1) \oplus \gamma^* \shO_V ( - D_2) \rar \dar  & \gamma^* I_Z \rar \dar & 0 \\
0 \rar & \shO_\CC  \rar \dar  & \shO_\CC \oplus \shO_\CC \rar \dar  & \shO_\CC \rar \dar & 0 \\
0 \rar & \shO_{P_1 + P_2}   \rar \dar  & \shO_{P_1} \oplus \shO_{P_2} \rar \dar  & \shO_{P_1\cap P_2} \rar \dar & 0 \\
& 0 & 0 & 0 &  
\end{tikzcd}\]
where $P_1 = \gamma^* D_1$ and $P_2 = \gamma^* D_2$ are divisors on $\CC$, and we used the identification $\gamma^* \shO_V = \shO_\CC$.
Note that the two left vertical sequences are exact because of the Zariski density of the image of $\gamma$, which consequently cannot be contained in 
any divisor on $V$ (a special example of Lemma \ref{pullback_free} above). By the Snake Lemma we obtain that the map in the upper right corner is also injective. 
\end{proof}

\subsubsection{Proof of Theorem~\ref{thm:MainGG}}
We first prove Item~\ref{item:GG1}.
Aiming for a contradiction, we assume that the image
$\gamma(\CC)$ is Zariski dense in $V$. 
We follow the set-up and notation of Proposition~\ref{prop:summary}.
By Proposition~\ref{prop:reduction}, we may assume that $V=Y \smallsetminus D$.

We may also assume that the morphism 
$$\sT_Y ( - \log D)  \overset{\psi}{\longrightarrow} {\sF_0}^{-1} \otimes \sF_1$$
is not injective, as otherwise by Lemma \ref{pullback_free} it 
follows that the composition of morphisms 
$$
\sT_{\CC} \to \gamma^* \sT_Y ( - \log D) \to   \gamma^*\big({\sF_0}^{-1} \otimes \sF_1\big) = \gamma^*\big(\sL^{-1}\otimes \sF_1\big)
    \hookrightarrow \gamma^*\big(\sL^{-1} \otimes \sE_1\big)
   $$
is also injective, contradicting Proposition~\ref{prop:BHHiggs}. By Lemma \ref{nonzero}, we also know that $\psi$ is not 
the zero map.
We define $\sG : = {\rm Im} (\psi)$, which therefore has generic rank one, and leads to a short exact sequence
$$
0  \longrightarrow \sK  \to   \sT_Y (- \log D)  \longrightarrow  \sG  \longrightarrow 0.
$$
Since $\sG$ injects in a torsion-free sheaf, it is torsion-free itself. Therefore $\sK$ is reflexive, hence an invertible sheaf since 
we are on a smooth surface.
Moreover, since it is saturated in $\sT_Y (-\log D)$, we must have 
$$\sG \simeq \sM \otimes  \mathcal{I}_Z,$$ 
where $\sM$ is a line bundle and $Z$ is a (possibly empty) $0$-dimensional subscheme of $Y$.  
It is standard that $Z$ is a local complete intersection.

Note that since $\sL \subseteq \sF_0$, we have an inclusion $\sG \subseteq \sL^{-1} \otimes \sE_1$. We claim that this induces an inclusion 
$$\gamma^* \sG \subseteq \gamma^*(\sL^{-1}\otimes \sE_1),$$
which in particular shows that $\gamma^* \sG$ it torsion free. To see this, note that the initial inclusion factors as a composition
$$\sM\otimes I_Z \hookrightarrow \sM \hookrightarrow \sL^{-1} \otimes \sE_1,$$
and the second map pulls back to an injective map by Lemma \ref{pullback_free}. It suffices then to have that the inclusion $I_Z \hookrightarrow \sO_V$ 
also pulls back to an injective map, and this is precisely the content of Lemma \ref{pullback_inclusion}.

Again by Lemma \ref{pullback_free}, the pullback sequence 
$$
0  \longrightarrow \gamma^* \sK  \longrightarrow   \gamma^* \sT_Y ( - \log D) \longrightarrow  \gamma^* \sG
 \longrightarrow 0
$$
is also exact. Since $\gamma^* \sG$ it torsion free, and 
the image of $\sT_{\CC}$ inside $\gamma^*(\sL^{-1}\otimes \sE_1)$ 
is zero by Proposition~\ref{prop:BHHiggs}, it follows that the map $\sT_{\CC} \to \gamma^* \sT_Y (-\log D)$ factors 
through $\gamma^* \sK$. 

Consider now the saturation $\sK^\prime$ of $\sK$ in $\sT_Y$, which defines a foliation on $Y$. Since the differential 
$\sT_\CC \to \gamma^* \sT_Y$ clearly factors through $\gamma^* \sK^\prime$ as well,  the image $\gamma(\CC)$ sits inside (or equivalently is tangent to) a leaf of this foliation. On the other hand, according to~\cite[Thm.~A]{PS15}, the pair $(Y, D)$ is of 
log general type. But this contradicts McQuillan's result \cite{McQ98} on the degeneracy of entire curves tangent to leaves of non-trivial foliations on surfaces of general type (cf. also \cite[Theorem~3.13]{Rousseau}), and more precisely its natural extension 
to the log setting as in El Goul \cite[Theorem~2.4.2]{El03}. This finishes the proof of Item~\ref{item:GG1}.

To prove Item~\ref{item:GG2}, just as in the proof of Theorem \ref{thm:MaxVar} let $V_0$ be the Zariski open subset in Lemma~\ref{lem:FiniteMap}
and $\wtilde V_0$ be the subset of $V_0$ over which the morphism $g$ is finite. 
We again claim that 
$$
 \Exc(V)  \subseteq V\smallsetminus \wtilde V_0.
$$
Assume on the contrary that there exists an entire curve $\gamma: \mathbb C \to V$ with 
$\gamma(\mathbb C)\cap \wtilde V_0 \neq \emptyset$. Then, by definition, 
the pull-back of the family $f$ via $\gamma$ has maximal variation. 
Since $\gamma (\CC)$ cannot be Zariski dense in $V$ by Item~\ref{item:GG1}, it is 
either a point, or it is dense in a quasi-projective curve $C$, which by Proposition~\ref{prop:reduction} 
can be assumed to be smooth. In the latter case, we thus obtain a smooth family of varieties of general type 
over $C$, with maximal variation. But then by \cite[Theorem~0.1]{Vie-Zuo01} we know that $C$ cannot be 
$\CC^*$, $\CC$, $\PP^1$ or an elliptic curve, which gives a contradiction.


\subsubsection{Proof of Corollary~\ref{cor:MainBH}.}
According to Proposition~\ref{prop:reduction} (Item~\ref{item:reduce2}), it is enough to show that there cannot be 
algebraically nondegenerate holomorphic maps $\gamma \colon \CC \to V$, where $V$ is a smooth quasi-projective variety of dimension $1$ or $2$ with a generically finite map $V\to  P_h$. If $\dim V = 2$, this follows from 
Theorem~\ref{thm:MainGG}, Item~\ref{item:GG1}. If $\dim V = 1$, it follows again from \cite[Theorem~0.1]{Vie-Zuo01}, 
as explained at the end of the proof of  Theorem~\ref{thm:MainGG}.

\smallskip

\begin{remark}\label{rem:dim1}
We note that Proposition~\ref{prop:BHHiggs} gives an alternative proof of \cite[Theorem~0.1]{Vie-Zuo01}, 
since it shows that a quasi-projective variety $V$ of dimension one is hyperbolic if it supports a birationally non-isotrivial 
smooth family of projective varieties whose geometric generic fiber admits a good minimal model. 
This is because, in this case, the map $\sT_V \to (\sL^{-1}\otimes \sE_1)|_V$ induced by 
$\sF_0 \to \sF_1\otimes \Omega^1_Y ( \log D)$ as in Proposition \ref{prop:summary} is an injection, as the latter map is injective by Lemma \ref{nonzero}. 
\end{remark}

\section{Appendix: Generic freeness and construction of sections}
\label{sect:appendix}

This is a technical appendix verifying that the sections needed in order to perform the Hodge module and Higgs bundle constructions in
\S\ref{sect:HodgeModules} can indeed be produced even after a birational modification ensuring that the singular locus of these Hodge theoretic objects has simple normal crossings. This is stated in \cite[Lemma 5.4]{Vie-Zuo03a} when the fibers of the family have semiample canonical bundle, and in \cite[\S2.2]{PS15} in general, but in both references the concrete details are not included. It turns out that they are somewhat technical, and therefore worth recording;  however, we emphasize that all the ingredients needed for the proof can 
be found in \cite{Vie-Zuo03a}, only one technical addition being needed when the canonical bundle of the fibers is not assumed to be 
semiample.

What we are aiming for is Proposition \ref{section_SNC} below. For its statement and proof, the starting point is the following generic freeness statement. We consider a smooth family $f_{\tilde U}\colon \tilde U\to \tilde V$  with projective fibers, whose geometric generic fiber admits a good minimal model, and with $\tilde U$ and $\tilde V$ smooth and quasi-projective. We assume that  $f_{\tilde U}$ has maximal variation.

 \begin{proposition}\label{thm:ggrefine}
With the assumptions above, there exist a smooth birational model $V\to \tilde V$, a smooth projective compactification $Y$ of $V$ with $D=Y\smallsetminus V$ a simple normal crossings divisor, an algebraic fiber space $f\colon X\to Y$, smooth over $V$, with $X$ smooth projective and  $f^{-1}(D)$ a simple normal crossings divisor, as well as an ample line bundle $\sL$  and an effective divisor 
$D_Y\ge D$ on $Y$, such that 
\[f_* \omega^m_{X/Y} \otimes \sL(D_Y)^{-m}\]
is generated by global sections over $V$ for all $m$ sufficiently large and divisible.  
Moreover, if the fibers of $f_{\tilde U}$ have semiample canonical bundle, then 
\[\omega^m_{X/Y}\otimes f^* \sL(D_Y)^{-m} \]
is also generated by global sections over $U=f^{-1}(V)$.
\end{proposition}

When the fibers of the family have semiample canonical bundle, this is nothing else but \cite[Prop. 4.1 and Cor. 4.3]{Vie-Zuo03a}. In the general case the proof is identical, based on Viehweg's fiber product trick and the mild reduction of Abramovich-Karu, except in one step we need to replace the use of weak positivity by that of the following analytic extension theorem of Berndtsson, P\u aun and Takayama, as stated by Cao \cite[Thm. 2.10]{Cao16}:\footnote{We are stating more precisely what is the locus over which global generation holds, but this is an immediate consequence of the proof in \emph{loc. cit.}} 

\begin{theorem}\label{thm:Cao}
Let $p:X\to Y$ be an algebraic fiber space between smooth projective varieties, and let $\sM$ be a line bundle on $X$ with a singular metric $h$ such that $i\Theta_{h}(\sM)\ge 0$ in the sense of currents. Let $\sB$ be a very ample line bundle on $Y$ such that the global sections of $\sB\otimes \omega_Y^{-1}$ separate $2n$-jets,  where $n$ is the dimension of $Y$, and let $V\subseteq Y$ be a Zariski open set such that $p$ is flat over $V$ and $h^0(X_y, \omega^k_{X/Y}\otimes \sM|_{X_y})$ is constant over $y\in V$, for some positive integer $k$.  Assume also  that the multiplier ideal $\mathcal J(h^{\frac{1}{k}}|_{X_y})=\sO_{X_y}$ for $y\in V$.
Then
\[f_*(\omega^k_{X/Y}\otimes \sM)\otimes \sB\]
is globally generated over $V$.
\end{theorem}

The idea of using this ingredient as a substitute for weak positivity is due to Y. Deng (see also \cite{Deng}), whom we thank for allowing us to use it here. We give the proof including all the details from \cite{Vie-Zuo03a}, as some are also necessary for the proof of Proposition \ref{section_SNC}.

\begin{proof}[Proof of Proposition \ref{thm:ggrefine}] 
We use the theory of mild morphisms; for the definition, and a discussion of the relevant properties, please see \cite[\S2]{Vie-Zuo03a}.
Using the mild reduction procedure of Abramovich-Karu (see \cite[Lem. 2.3]{Vie-Zuo03a}), there exists $f_U\colon U\to V$, a birational model of the original family $f_{\tilde U}\colon \tilde U\to \tilde V$ with $U$ and $V$ smooth, which fits into a diagram
\[\begin{tikzcd}
U\rar{\subseteq}\dar{f_U}&W \dar{f_W} & \lar W'   \dar & \lar{\sigma_W} Z_W\dar{g_W} & \lar{\rho_W} W'' \rar{\delta_W}\dar{f''_W}& Z'_W\dar{g'_W}\\
V\rar{\subseteq}&Y  & \lar{\tau} Y' &\lar{=} Y'&\lar{=} Y'\rar{=} &Y',
\end{tikzcd}
\] 
where $\tau$ is a finite morphism with $Y'$ smooth, branched over a simple normal crossing divisor $\Delta_\tau$, $W'$ is the normalization of the main component of $W\times_Y Y'$, $\sigma_W$ is a resolution of $W'$ with centers in its singular locus, $\rho_W$ and $\delta_W$ are birational with $W''$ smooth, and $g'_W$ is mild. By taking further resolutions, we  
can assume that $\Delta_\tau+D$ and $f^{-1}_W(\Delta_\tau+D)$ are simple normal crossings divisors, where $D=Y\smallsetminus V$. 
By possibly composing $\tau$ with a Kawamata covering, we are allowed to assume that $Y'$ is smooth, and hence $W'$ is normal with rational singularities (see for instance \cite[Lem. 2.1]{Viehweg83}).

Since $f_{\tilde U}$ is of maximal variation, with geometric generic fiber admitting a good minimal model, by construction so is $g_W$,  hence by a well-known result of Kawamata \cite{Kawamata} we know that $[\det] g_{W*} \omega^v _{Z_W/Y'}$ is a big line bundle for some integer 
$v$ sufficiently large and divisible. Here and in what follows we use $[\bullet]$ to denote the reflexive hull of the corresponding operator on sheaves.

Fix an ample line bundle $\sA_Y$ on Y. Pick $k_0$ large enough so that $\sA=\sA_Y^{k_0}(-D)$ is also ample, $\tau^*\sA$ is very ample, and the global sections of $\tau^*\sA\otimes \omega^{-1}_{Y'}$ separate $2n$-jets. Then, by \cite[Cor. 2.4(ix)]{Vie-Zuo03a}, we have 
$$\big([\det] g_{W*} \omega^v _{Z_W/Y'} \big)^{N_v}=\tau^*\sA(D'+D)$$ 
for some positive integer $N_v$, where $D'$ an effective divisor on $Y$. 

Another input we need is the fact that the quantity
\[e(\omega^v_{W_y})=\sup\{\frac{1}{\textup{lct}(B)}~|~ B\in |\omega^v_{W_y}|\},\]
where $\textup{lct}(B)$ is the log canonical threshold of $B$, is upper semicontinuous as a function of $y \in V$. This is simple consequence of the standard lower semicontinuity of the log canonical threshold of divisors that are relative for smooth proper morphisms, combined with 
the invariance of plurigenera. It follows that there exists a positive integer $C$ such that 
\[e(\omega^v_{W_y})< Cv\]
uniformly for every $y\in V$.

We now take $r=C(C+1)vN_vr_0$, where $r_0=\textup{rank}(g_{W*} \omega^v _{Z_W/Y'})$. We obtain a new family 
$f\colon X\to Y$ by taking $X=W^{(r)}$, a desingularization of the main component of the $r$-th fiber product $W\times_Y\dots\times_Y W$. As always, we are allowed to assume that $f^{-1}(\Delta_\tau+D)$ is normal crossing. Completely similarly to the process for $f_W$, we can fit $f$ into a reduction diagram 
\[\begin{tikzcd}
X \dar{f} & \lar X'   \dar & \lar{\sigma} Z\dar{g} & \lar{\rho} X'' \rar{\delta}\dar{f''}& Z'\dar{g'}\\
Y  & \lar{\tau} Y' &\lar{=} Y'&\lar{=} Y'\rar{=} &Y',
\end{tikzcd}
\] 
where $X'$ is the normalization of the main component of $X\times_Y Y'$ (so that $X'$ has rational singularities), $\sigma$ is a resolution of $X'$ with centers in the singular locus, $\rho$ and $\delta$ are birational, with $X''$ smooth, and $Z' =Z'_W\times_Y\cdots\times_YZ'_W$ with the morphism $g'$ induced by $g'_W$. Since $g'_W$ is mild, we know that $g'$ is also mild (see \cite[Lem. 2.2(ii)]{Vie-Zuo03a}). 

Now by \cite[Cor. 2.4(vii)]{Vie-Zuo03a} we have
\[g'_{W*} \omega^v_{Z'_W/Y'}  \simeq g_{W*} \omega^v_{Z_W/Y'},\]
and both sheaves are reflexive. Hence, by flat base change and the projection formula, since $g'_*\omega_{Z'/Y'}^v$ is 
also reflexive, we get
$$g'_*\omega_{Z'/Y'}^v \simeq [\bigotimes^r] g'_{W*} \omega^v_{Z'_W/Y'}  \simeq [\bigotimes^r] g_{W*} \omega^v_{Z_W/Y'}.$$
Thanks again to \cite[Cor. 2.4(vii)]{Vie-Zuo03a}, we also have $g'_*\omega_{Z'/Y'}^v \simeq g_*\omega_{Z/Y'}^v$. On the other hand, there is a natural morphism 
\[[\det] g_{W*} \omega^v _{Z_W/Y'} \longrightarrow [\bigotimes^{r_0}] g_{W*} \omega^v_{Z_W/Y'},\]
which splits locally over $V'=\tau^{-1}(V)$ (since $g_W$ is smooth over $V'$, so is $g$). Putting everything together, we obtain an injective morphism 
\[\tau^*\sA(D'+D)^{C(C+1)v}=\big([\det] g_{W*} \omega^v _{Z_W/Y'} \big)^{C(C+1)vN_v}\longrightarrow g_*\omega_{Z/Y'}^v,\]
which also splits locally over $V'$.  This corresponds to an effective divisor 
$$\Gamma \in |\omega^v_{Z/Y'}\otimes g^* \tau^*\sA(D'+D)^{-C(C+1)v}|$$ 
which does not contain the fiber $Z_{y'}$ for every  $y'\in V'$. 

Since $Z_{y'}=W_y^r=W_y\times\cdots\times W_y$, using the bound $e(\omega^v_{W_y})< Cv$ and \cite[Cor. 5.21]{Viehweg95}, we have 
$$\textup{lct}(\Gamma|_{Z_{y'}})> \frac{1}{Cv}$$  
for every $y'\in V'$, where $y=\tau(y')$. 

For all $k>0$, we can apply Theorem \ref{thm:Cao} to the line bundle 
$$\sM = \omega^{kv}_{Z/Y'}\otimes g^* \tau^*\sA(D'+D)^{-kC(C+1)v},$$ 
with the natural singular metric induced by the effective divisor $k\Gamma$, with $\sB = \tau^* \sA$, using the invariance of plurigenera and the fact that 
 $\frac{1}{Cv}\Gamma|_{X_{y'}}$ is klt for all $y'\in V'$.  Consequently
\[g_*\omega^{k(C+1)v}_{Z/Y'}\otimes\tau^*\sA(D'+D)^{-kC(C+1)v}\otimes \tau^*\sA\]
is globally generated over $V'$.

By \cite[Lem. 3.2]{Viehweg83}, we have a natural morphism
\begin{equation}\label{eq:nfred}
g_*\omega^{k(C+1)v}_{Z/Y'}\longrightarrow \tau^*f_*\omega^{k(C+1)v}_{X/Y}
\end{equation}
which is an isomorphism over $V'$.
Since $\tau$ is finite, we can apply the projection formula to get a morphism 
\[\bigoplus_i \tau_*\sO_{Y'} \longrightarrow f_*\omega^{k(C+1)v}_{X/Y}\otimes\tau_*\sO_{Y'}\otimes\sA(D+D')^{-kC(C+1)v}\otimes\sA\] 
which is surjective over $V$. Now we pick $k$ sufficient large so that $\tau_*\sO_{Y'}\otimes \sA^{k(C+1)v-1}$ is globally generated. Therefore 
\[ f_*\omega^{K(C+1)v}_{X/Y}\otimes\tau_*\sO_{Y'}\otimes\sA^{-k(C-1)(C+1)v}\big(-kC(C+1)v(D+D')\big)\]
is generated by global sections over $V$, and so via the trace map so is
\[f_*\omega^{k(C+1)v}_{X/Y}\otimes\sA^{-k(C-1)(C+1)v}\big(-kC(C+1)v(D+D')\big).\]
Setting $m=K(C+1)v$, $\sL=\sA^{C-1}$ and $D_Y=C(D+D')$, we obtain the first statement.

The second statement follows immediately from the first, noting that the 
assumption implies that  the natural map
\[f^*f_*\omega^m_{X/Y}\longrightarrow \omega^m_{X/Y}\] 
is surjective over $U$. 
\end{proof}

\begin{remark}\label{section_space}
Recalling that $g'_*\omega_{Z'/Y'}^{m} \simeq g_*\omega_{Z/Y'}^{m}$, 
the proof of the proposition above shows more precisely that $f_* \omega^m_{X/Y} \otimes \sL(D_Y)^{-m}$ is generated over $V$ 
by sections belonging to the subspace $\V_m$ defined as the image of the natural map
\[H^0\big(Y',g'_*\omega^m_{Z'/Y'}\otimes\tau^*\sL(D_Y)^{-m}\big)\longrightarrow H^0\big(Y, f_*\omega^{m}_{X/Y}\otimes\sL(D_Y)^{-m}\big)\]
induced by ($\ref{eq:nfred}$).
\end{remark}

\begin{proposition}\label{section_SNC}
With notation as in the statement and proof of Theorem \ref{thm:ggrefine} and Remark \ref{section_space}, let $S\ge \Delta_\tau$ be an effective divisor and $\mu\colon \tilde  Y\to Y$ a log resolution of $(Y, D+S)$ with centers in the singular locus of $D+S$. Then for every 
$s\in \V_m$, there exists a closed subset $T\subset \tilde Y$ of codimension at least $2$,  a birational model $\tilde f\colon  \tilde  X\to 
\tilde Y$ of $f$ with $\tilde X$ smooth and projective, and a section 
\[\tilde s\in H^0\big(\tilde Y_0, \tilde f_*\omega^m_{\tilde X/\tilde Y}\otimes \mu^*\sL(D_Y)^{-m}\big),\] 
with $\tilde Y_0=\tilde Y\smallsetminus T$, such that 
\[\tilde s|_{\mu^{-1}(V\smallsetminus S)}=\mu^*(s|_{V\smallsetminus S}).\]
\end{proposition}
\begin{proof}
By the definition of $\V_m$, we can lift $s$ to a section 
\[s'\in H^0\big(Y',g'_*\omega^{m}_{Z'/Y'}\otimes\tau^*\sL(D_Y)^{-m}\big).\]
We set $\tilde Y''$ to be the normalization of the main component of $Y'\times_Y \tilde Y$ and $\tilde\tau'\colon \tilde Y''\to \tilde Y$ the induced finite map. We compose this with a desingularization $\mu' \colon \tilde Y'\to \tilde Y''$, and get a birational map $\tilde\mu'\colon \tilde Y'\to Y'$.  We then take $\tilde X$ to be a desingularization of the main component of 
$X\times_Y\tilde Y$, so that the induced morphism $\tilde f\colon \tilde X\to \tilde Y$ is a birational model of $f$. We obtain a mild reduction diagram for the new family 
$\tilde f$
\[\begin{tikzcd}
\tilde X \dar{\tilde f} & \lar \tilde X'   \dar{\tilde f'} & \lar{\tilde\sigma} \tilde Z\dar{\tilde g} & \lar{\tilde\rho} \tilde X''_1 \rar{\tilde\delta}\dar{\tilde f''}& \tilde Z'\dar{\tilde g'}\\
\tilde Y  & \lar{\tilde\tau} \tilde Y' &\lar{=} \tilde Y'&\lar{=} \tilde Y'\rar{=} &\tilde Y',
\end{tikzcd}
\]
where $\tilde\tau=\tilde\tau'\circ \mu'$, a generically finite morphism, $\tilde Z'=Z'\times_Y'\tilde Y'$ and $\tilde g'$ is the induced mild morphism, 
and $\tilde \sigma$, $\tilde\rho$ and $\tilde\delta$ are similar to those in the mild reduction diagram for $f$. 

In particular we have a Cartesian diagram 
\[\begin{tikzcd}
 Z'\dar{g'}&  \tilde Z'  \lar{\tilde\nu'} \dar{ \tilde g'}\\
Y' &  \tilde Y' \lar{\tilde\mu'}.
\end{tikzcd}
\]
By the mildness of the vertical morphisms we have that $\tilde Z'$ is normal with rational singularities, and so 
\[\tilde\nu'^*\omega^m_{Z'/Y'} \simeq \omega^m_{\tilde Z'/\tilde Y'}.\]
Setting  $\tilde \sL=\mu^*\sL$ and $\tilde D_{\tilde Y}=\mu^*D_Y$, we thus conclude that 
$s'$ lifts to a section 
$$\tilde s'=\tilde\mu'^*s' \in H^0\big(\tilde Y', \tilde g'_*\omega^{m}_{\tilde Z'/\tilde Y'}\otimes\tilde\tau^*\tilde\sL(\tilde D_{\tilde Y})^{-m}\big).$$

Since $\tilde \tau$ is generically finite and branched over the normal crossing divisor $\mu^{-1}(\Delta_\tau)$, there exists a subset $T\subset \tilde Y$ of codimension at least $2$ 
such that $\tilde\tau|_{\tilde Y_0}$ is finite and flat, where $\tilde Y_0=\tilde Y\smallsetminus T$. We are also allowed to assume that $\tilde f^{-1}(\mu^{-1}(S))$ is a simple normal crossing divisor, and hence so is $\tilde f^{-1}(\mu^{-1}(\Delta_\tau))$, by taking further blow-ups if necessary.
Setting $\tilde Y'_0=\tilde\tau^{-1}(\tilde Y_0)$, we deduce that $\tilde X'_0=\tilde f'^{-1}Y'_0$ is normal with rational singularities. (For all of these statements, see e.g. \cite[Lem. 2.1]{Viehweg83}.) Therefore, thanks to \cite[Lem. 3.2]{Viehweg83}, there is a morphism
\begin{equation} \label{eq:nfred1}
\tilde g'_*\omega^{m}_{\tilde Z'/\tilde Y'}|_{\tilde Y'_0}\longrightarrow \tilde\tau^*\tilde f_*\omega^{m}_{\tilde X/\tilde Y}|_{\tilde Y'_0},
\end{equation}
which is identical to \eqref{eq:nfred} over $\tau^{-1}(V\setminus S)$ (as $\mu$ is the identity over $V\setminus S$).
This in turn induces
\[ \tilde\tau_*\tilde g'_*\omega^{m}_{\tilde Z'/\tilde Y'}|_{\tilde Y_0}\longrightarrow \tilde\tau_*\tilde\tau^*\tilde f_*\omega^{m}_{\tilde X/\tilde Y}|_{\tilde Y_0}\longrightarrow \tilde f_*\omega^{m}_{\tilde X/\tilde Y}|_{\tilde Y_0},\]
where the last morphism is induced by the trace map $\tilde \tau_*\sO_{\tilde Y'_0}\to \sO_{\tilde Y_0}$. Finally, we conclude the existence of a morphism
\[\eta: \tilde\tau_*\tilde g'_*\omega^{m}_{\tilde Z'/\tilde Y'}\otimes \tilde\sL(\tilde D_{\tilde Y})^{-m} |_{\tilde Y_0}\to \tilde f_*\omega^{m}_{\tilde X/\tilde Y}\otimes \tilde\sL(\tilde D_{\tilde Y})^{-m} |_{\tilde Y_0}\]
and define 
\[\tilde s=\eta(\tilde s'|_{\tilde Y'_0})\in H^0\big(\tilde Y_0, \tilde f_*\omega^l_{\tilde X/\tilde Y}\otimes \mu^*\sL(D_Y)^{-l}\big).\]
Since \eqref{eq:nfred} and \eqref{eq:nfred1} are isomorphisms over $\tau^{-1}(V\smallsetminus S)$, we have 
$$\tilde s|_{\mu^{-1}(V\smallsetminus S)}=\eta(\tilde s')|_{\mu^{-1}(V\smallsetminus S)}=\mu^*(s|_{V\smallsetminus S}).$$
\end{proof}

\bibliography{bibliography/general}



\end{document}